\documentclass[11pt,reqno]{amsart}
\usepackage[english]{babel}
\usepackage{amssymb,enumerate,bbm,amsmath}
\usepackage{url}
\usepackage[
linktocpage=true,colorlinks=true,linkcolor=blue,citecolor=magenta,urlcolor=blue]{hyperref}
\numberwithin{equation}{section}
\allowdisplaybreaks
\usepackage[margin=1in]{geometry}
\newtheorem{theo}{Theorem}[section]
\newtheorem{pro}[theo]{Proposition}
\newtheorem{lem}[theo]{Lemma}
\newtheorem{cor}[theo]{Corollary}

\theoremstyle{remark}
\newtheorem{rem}[theo]{Remark}
\newtheorem*{ack}{Acknowledgments}
\renewcommand{\(}{\left(}
\renewcommand{\)}{\right)}

\newcommand{\R}{\mathbb{R}}

\renewcommand{\H}{\mathbb{H}}

\renewcommand{\a}{\alpha}
\renewcommand{\b}{\beta}
\newcommand{\g}{\gamma}
\renewcommand{\d}{\delta}

\renewcommand{\k}{\kappa}

\newcommand{\D}{\Delta}

\newcommand{\ra}{\rightarrow}

\newcommand{\mrm}{\mathrm}

\begin{document}

\title[Contraction of surfaces]{Contraction of surfaces in hyperbolic space and in sphere}
\author[Y. Hu]{Yingxiang Hu}
\address{School of Mathematical Sciences, Beihang University, Beijing 100191, P.R. China}
\email{huyingxiang@buaa.edu.cn}
\author[H. Li]{Haizhong Li}
\address{Department of Mathematical Sciences, Tsinghua University, Beijing 100084, P.R. China}
\email{lihz@tsinghua.edu.cn}
\author[Y. Wei]{Yong Wei}
\author[T. Zhou]{Tailong Zhou}
\address{School of Mathematical Sciences, University of Science and Technology of China, Hefei 230026, P.R. China}
\email{yongwei@ustc.edu.cn}
\email{ztl20@ustc.edu.cn}


\begin{abstract}
In this paper, we consider the contracting curvature flows of smooth closed surfaces in $3$-dimensional hyperbolic space and in $3$-dimensional sphere. In the hyperbolic case, we show that if the initial surface $M_0$ has positive scalar curvature, then along the flow by a positive power $\alpha$ of the mean curvature $H$, the evolving surface $M_t$ has positive scalar curvature for $t>0$. By assuming $\alpha\in [1,4]$, we can further prove that $M_t$ contracts a point in finite time and  become spherical as the final time is approached. We also show the same conclusion for the flows by powers of scalar curvature and by powers of Gauss curvature provided that the power $\a\in [1/2,1]$.

In the sphere case, we show that the flow by a positive power $\alpha$ of mean curvature contracts strictly convex surface in $\mathbb{S}^3$ to a round point in finite time if $\alpha\in [1,5]$. The same conclusion also holds for the flow by powers of Gauss curvature provided that the power $\a\in [1/2,1]$.
\end{abstract}

\maketitle

\section{Introduction}\label{sec:1}
Let $\R^3(c)$ ($c=0,1,-1$) be the standard model of simply connected space form, i.e., when $c=0$, $\R^3(0)=\R^3$, when $c=1$, $\R^3(1)=\mathbb{S}^3$ and when $c=-1$, $\R^3(-1)=\H^3$. Let $M_0$ be a smooth closed surface in $\R^3(c)$, given by a smooth immersion $X_0: M^2\to\R^3(c)$. We consider the contracting curvature flow of closed surfaces starting at $M_0$ in $\R^3(c)$, which is a family of smooth immersions $X:M^2\times [0,T)\to \R^3(c)$ satisfying
\begin{align}\label{1.1}
\left\{\begin{aligned}\frac{\partial}{\partial t} X(x,t)=&-F(x,t) \nu(x,t),\\
X(x,0)=&X_0(x),
\end{aligned}\right.
\end{align}
where $F$ is a smooth, symmetric function of the principal curvatures $\kappa=(\kappa_1,\kappa_2)$ of the evolving surface $M_t=X(M,t)$, and $\nu$ is the outward unit normal of $M_t$.

\subsection{Background}
There are many papers which consider the contraction of hypersurfaces in Euclidean space $\R^{n+1}$ under the flow \eqref{1.1}. In his foundational work \cite{Huisken1984}, Huisken proved that any compact strictly convex hypersurface in Euclidean space, evolving by the mean curvature flow (i.e., $F$ is given by the mean curvature $H$), will become spherical as it shrinks to a point. Later, Chow \cite{Chow1985} proved the same conclusion for flow \eqref{1.1} with speed $F$ given by $n$th root of the Gauss curvature $K$. He also
proved a result for flow by the square root of the scalar curvature \cite{Chow1987}, but in that case a stronger assumption than convexity was required for the initial hypersurface. These results have been generalized by Andrews \cite{And94-a,Andrews2007,Andrews2010} to a large class of speed functions $F$ which are homogeneous of degree one of the principal curvatures, and satisfy certain natural concavity conditions.

For speed function $F$ with higher homogeneity, the analysis of the flow \eqref{1.1} becomes much more difficult. In this direction, several works have treated such flows in a special case of surfaces in $3$-dimensional Euclidean space $\mathbb{R}^3$:  Andrews \cite{Andrews1999} proved that any strictly convex surface in $\mathbb{R}^3$ will shrink to a round point along the Gauss curvature flow (i.e., the flow \eqref{1.1} with $F=K$), which affirmatively resolves the famous Firey's conjecture. The key step in the proof is a curvature pinching estimate, which says that the ratio of the largest principal curvture to the smallest one can be controlled by its initial value. This was proved by applying the maximum principle to a suitably chosen function of the principal curvatures $\k_1,\k_2$ of $M_t$. This idea has been explored further in the works \cite{Andrews-Chen2012,Schnurer2005,Schulze2006}.  In particular, Schulze and Schn\"urer\cite{Schulze2006} proved that the flow by powers of mean curvature in $\R^3$ (i.e., $F=H^{\alpha}$) constracts convex surface to a round point provided that $\alpha\in [1,5]$; Andrews and Chen \cite{Andrews-Chen2012} proved that the flow by powers of Gauss curvature in $\R^3$ (i.e., $F=K^{\alpha}$) contracts convex surface to a round point provided that $\a\in [1/2,1]$. In both papers \cite{Andrews-Chen2012,Schulze2006}, the authors considered the following quantity
\begin{equation}\label{s1:G1}
  G(x,t)=\frac{(\k_1-\k_2)^2}{\k_1^2\k_2^2}F^{2}
\end{equation}
and showed that the spatial maximum of $G$ is monotone non-increasing along the flow for certain range of the power $\a$, which allows the authors to prove the crucial curvature pinching estimate. In the higher dimensional case, the flow by powers of Gauss curvature has been well studied. The complete picture for this flow has been captured by the combined works \cite{AGN16,Brendle-Choi-Daskalopoulps2017,GN17}. For general curvature flows with high powers homogeneous speed functions in Euclidean space, in order to show that closed convex hypersurface shrinks to a round point, a strong curvature pinching condition on the initial hypersurface is always needed, see \cite{Alessandroni-Sinestrari2010,AM12,Schulze2006}.

In the hyperbolic space, the understanding of the flow \eqref{1.1} is less complete. Huisken \cite{Huisken1986} proved that the mean curvature flow contracts compact hypersurface with principal curvatures satisfying $\kappa_iH>n$ ($\forall~i=1,\cdots,n$) in hyperbolic space $\mathbb{H}^{n+1}$ to a round point. Andrews \cite{Andrews1994} considered a large class of  fully nonlinear flows which doesn't include the mean curvature flow. It is shown that any initial compact hypersurface  in hyperbolic space which is horospherically convex (i.e., $\kappa_i>1$ for all $i=1,\cdots,n$) can be deformed to a round point along the flow. A typical example included in \cite{Andrews1994} is the flow by shifted harmonic mean curvature $F=(\sum_{i=1}^n(\kappa_i-1)^{-1})^{-1}$. In \cite{Yu16}, Yu studied the contracting flows in hyperbolic space for a general class of homogeneous of degree one speed functions, using a similar argument as in Gerhardt \cite{Ger15} for curvature flows in the sphere. Recently, Andrews and Chen \cite{Andrews-Chen2017} proved the smooth convergence of the mean curvature flow for hypersurfaces with positive Ricci curvature in hyperbolic space. In the special case of surfaces in $3$-dimensional hyperbolic space $\H^3$, they also studied the behavior of the flow \eqref{1.1} for surfaces with positive intrinsic scalar curvature $R=2(K-1)>0$. In particular, they proved that the mean curvature flow and scalar curvature flow (i.e., $F=K-1=R/2$) preserve the condition $R>0$ and evolve any closed surface with positive scalar curvature to a round point in finite time.
The negative curvature of the ambient space $\H^3$ produces terms in the evolution of the second fundamental form, which prevent the estimates in Euclidean setting from being applied in hyperbolic setting.  To overcome this difficulty,  different improving quantities have been used in \cite{Andrews-Chen2017} to obtain the curvature pinching estimate.

For the curvature flows in the sphere, Huisken \cite{Hui87} proved that for any initial hypersurface which satisfies a pointwise pinching condition, the mean curvature flow will contract the hypersurface to a point in finite time, or evolve for all time to a smooth totally geodesic hypersurface. Gerhardt \cite{Ger15} proved that any strictly convex hypersurface will be contracted to a round point in finite time along the flow in sphere if the speed function $F$ is concave and inverse concave with respect to the principal curvatures. In the special case of surfaces in $3$-dimensional sphere $\mathbb{S}^3$, Andrews \cite{And2002} optimised the choice of the fully nonlinear speed function to show that any surfaces with positive intrinsic curvature in sphere $\mathbb{S}^3$ can be deformed to either a round point in finite time, or to the great sphere in infinite time. McCoy \cite{McCoy2017} recently proved that the flow by any homogeneous of degree one speed function or by Gauss curvature can evolve strictly convex surfaces to a round point in finite time.


\subsection{Main results}
In this paper, we focus on the contracting curvature flow of surfaces in $3$-dimensional hyperbolic space $\mathbb{H}^3$ and in $3$-dimensional sphere $\mathbb{S}^3$. In the hyperbolic case, we assume that the initial surface has positive intrinsic scalar curvature. We will study the contraction of such surfaces along the flow \eqref{1.1} in $\mathbb{H}^3$ by powers of mean curvature, powers of scalar curvature and powers of Gauss curvature. Our first result states as follows:
\begin{theo}\label{main-theo-I}
Let $X_0:M^2 \ra \H^3$ be a smooth closed surface with positive scalar curvature in $\H^3$. Assume that either
\begin{itemize}
  \item[(i)] $F=H^{\a}$ with $\a \in [1,4]$ ; or
  \item[(ii)] $F=(K-1)^{\a}$ with $\a\in [1/2,1]$; or
  \item[(iii)] $F=K^{\a}$ with $\a\in [1/2,1]$.
\end{itemize}
Then there exists a unique solution $X: M^2 \times [0,T) \ra \H^3$ of the flow \eqref{1.1} on a maximum time interval $[0,T)$, where $T<\infty$. The surface $M_t$ has positive scalar curvature for each $t\in [0,T)$, and converges smoothly to a point $p\in \H^3$ as $t\ra T$. The solutions are asymptotic to a shrinking sphere as $t\ra T$ in the following sense: Let $\Theta(t,T)$ be the spherical solution of the flow with the same existence time $T$. Introducing geodesic polar coordinate system with respect to the point $p$, and writing the evolving surfaces $M_t$ as graphs of a function $u(x,t)$ on $\mathbb{S}^2$, then the rescaled function $u\Theta^{-1}$ is uniformly bounded and converges to $1$ in $C^{\infty}(\mathbb{S}^2)$ as $t\to T$.
\end{theo}

In the sphere case, we assume that the initial surface is strictly convex in $\mathbb{S}^3$. Then it lies strictly in a hemisphere of $\mathbb{S}^3$. We consider its contraction along the flow \eqref{1.1} in $\mathbb{S}^3$ by powers of mean curvature and powers of Gauss curvature.
\begin{theo}\label{main-theo-II}
Let $X_0:M^2 \ra \mathbb{S}^3$ be a smooth, closed and strictly convex surface in $\mathbb{S}^3$. Assume that either
\begin{itemize}
  \item[(i)] $F=H^{\a}$ with $\a \in [1,5]$ ; or
  \item[(ii)] $F=K^{\a}$ with $\a\in [1/2,1]$.
\end{itemize}
Then there exists a unique solution $X: M^2 \times [0,T) \ra \mathbb{S}^3$ of the flow \eqref{1.1} on a maximum time interval $[0,T)$, where $T<\infty$. The surface $M_t$ is strictly convex for each $t\in [0,T)$, and converges smoothly to a point $p\in \mathbb{S}^3$ as $t\ra T$. The solutions are asymptotic to a shrinking sphere as $t\ra T$ in the same sense as Theorem \ref{main-theo-I}.
\end{theo}

\subsection{Discussion on the proof}

The proof of our theorems mainly relies on the crucial curvature pinching estimate. In the hyperbolic case, instead of using \eqref{s1:G1} we consider the following quantity
\begin{align}\label{1.2}
G(x,t):=\frac{(\k_1-\k_2)^2}{(\k_1\k_2-1)^2}F^{2}
\end{align}
on the evolving surfaces $M_t$ (see \cite{Andrews-Chen2017} for the special case where $F=K-1$). Since the initial surface satisfies $K=\kappa_1\kappa_2>1$, we apply maximum principle to show that this condition is preserved, and thus the function $G$ is well-defined on $M_t$ for all $t\in [0,T)$. We will show that the spatial maximum of $G$ is monotone non-increasing along the flow \eqref{1.1} with speeds listed in Theorem \ref{main-theo-I}. The proof is by applying maximum principle to the evolution equation of $G$. The main advantage of this quantity $G$ is that there are no zero order terms in its evolution equation, provided that the speed function $F$ is a homogeneous function of the principal curvatures or $F$ is a power of the scalar curvature. Thus, the monotonicity of $G$ along the flow reduces to the non-positivity of the gradient terms at the critical points. This would be the most technique part in the proof: We will first derive in \S \ref{sec:2-3} a general formula for the gradient terms at the spatial critical point of $G$ for general speed function $F$. Then in \S \ref{sec:3} - \S \ref{sec:5}, we treat the three speed functions listed in Theorem \ref{main-theo-I} separately. By careful calculation, we eventually find the sufficient condition for the power $\alpha$ such that the gradient terms are non-positive, and therefore we conclude that along the flow the spatial maximum of $G$ would be monotone non-increasing in time. We should point out that the analysis of the gradient terms here is much more complicated due to the negative curvature of the ambient space $\mathbb{H}^3$. In fact, the coefficients in front of $(\nabla_1 h_{11})^2$ and $(\nabla_2 h_{22})^2$ are not homogeneous in $\k_1,\k_2$, which is quite different from the Euclidean case considered in \cite{Schulze2006}.

To prove the curvature pinching estimate in $\mathbb{S}^3$, we use the same test function $G$ defined in \eqref{s1:G1} as in Euclidean case \cite{Andrews-Chen2012,Schulze2006}. The evolution equation of $G$ has both the zero order terms and gradient terms. Since the function $G$ is the same one as in Euclidean case used in \cite{Andrews-Chen2012,Schulze2006}, the analysis of the gradient terms would be similar with the Euclidean case. We will apply the estimate of gradient terms in \cite{Schulze2006} directly here for the flow by powers of mean curvature in the sphere. However, the estimate of the gradient terms in \cite{Andrews-Chen2012} for the flow by powers of Gauss curvature is carried out using the Gauss map parametrization of the flow: The flow of strictly convex closed surfaces in Euclidean space $\mathbb{R}^3$ is equivalent to a scalar parabolic equation on the sphere $\mathbb{S}^2$ for the support function of the evolving surfaces. This parametrization is not available in the sphere case. Instead we prove our estimate using the calculation on the evolving surfaces directly. Moreover, due to the positive curvature of the ambient space $\mathbb{S}^3$ we have a good sign for the zero order terms. Thus the maximum principle can be applied to give the monotonicity of $G$. The details will be given in \S \ref{sec:S}.

In the last section, we describe the convergence of the solution to a point and of the rescaled solution to a sphere as stated in Theorem \ref{main-theo-I}. We only focus on the flow in $\mathbb{H}^3$ by powers of mean curvature, using the idea in the work by \cite{Schulze2006} and \cite{Andrews-Chen2017}.
The main difficulty is that for $\a> 1$ the parabolic operator involved in the evolution equation of the rescaled mean curvature $\tilde{H}$ has a coefficient $\alpha \tilde{H}^{\alpha-1}$ in the second order part and it becomes degenerate for $\tilde{H}$ sufficiently small. Thus we can not apply the parabolic Harnack inequality to get an estimate from below for $\tilde{H}$. To overcome this problem, we write the evolution equation satisfied by $\tilde{H}$ as a porous medium equation, and apply a result of DiBenedetto and Friedman \cite{DiBenedetto-Friedman1985} to get the H\"{o}lder continuity of $\tilde{H}$.
The proof for the remaining flows is similar.  See \S \ref{sec:6} for the details.

\begin{ack}
The authors would like to thank Professor Ben Andrews for suggestions on the curvature pinching in the hyperbolic case. The first author was
supported by China Postdoctoral Science Foundation (No.2018M641317). The second author was supported by NSFC grant No.11671224, 11831005 and NSFC-FWO grant No.1196131001. The third author was supported by Discovery Early Career Researcher Award DE190100147 of the Australian Research Council.
\end{ack}

\section{Evolution equations}\label{sec:2}
In this section, we collect some basic evolution equations along the flow \eqref{1.1}, and then derive the evolution equation of the quantity $G$ defined in \eqref{1.2}.

\subsection{Notations}
First, we fix the notations we will use in the paper.  We denote by $g=(g_{ij})$ and $h=(h_{ij})$ the induced metric and the second fundamental form of the surface $M_t$, respectively. Then the Weingarten tensor $\mathcal{W}=(h_i^j)=(g^{jk}h_{ik})$. The eigenvalues $\k_1,\k_2$ of $\mathcal{W}$ are called the principal curvatures of $M_t$. In the flow \eqref{1.1}, the speed function $F=F(\mathcal{W})$ is a smooth symmetric function of the Weingarten tensor $\mathcal{W}=(h_i^j)$ of the evolving surface $M_t=X_t(M)$, and $\nu$ is the outward unit normal of $M_t$. Equivalently, $F=F(\mathcal{W})=f(\kappa(\mathcal{W}))$, where $f$ is a smooth symmetric function of $2$-variables, and $\kappa(\mathcal{W})=(\k_1,\k_2)$ denotes the eigenvalues of $\mathcal{W}$.

The derivatives of $F$ with respect to the components of $\mathcal{W}=(h_i^j)$ and those of $f$ with respect to $\k_i$ are related in the following way (see e.g. \cite{Andrews2007}): Let $\dot{f}^i$ and $\ddot{f}^{ij}$ denote the derivatives of $f$ with respect to $\kappa_i$. If $A$ is diagonal and $B$ is a symmetric matrix, then the first derivative of $F$ is given by
\begin{align}\label{2.F-1st}
\dot{F}^{kl}(A)=\dot{f}^k(\k(A))\d^{kl};
\end{align}
and if $A$ has distinct eigenvalues, then the second derivatives of $F$ in direction $B$ is given by
\begin{align}\label{2.F-2nd}
\ddot{F}^{kl,rs}B_{kl} B_{rs}=\ddot{f}^{kl}(\k(A))B_{kk} B_{ll}+2\sum_{k<l} \frac{\dot{f}^k-\dot{f}^l}{\k_k-\k_l}B_{kl}^2.
\end{align}
The last term is interpreted as a limit if $\k_k=\k_l$. Since $F$ is symmetric, we can assume that at each point $(p,t)\in M \times [0,T)$, the principal curvatures satisfy $\k_1 \geq \k_2$.

\subsection{Evolution equations}

Along the flow \eqref{1.1} in $\R^3(c)$, the speed function $F$ satisfies the evolution equations (see \cite{Andrews1994}):
\begin{equation}\label{s2:speed}
  \frac{\partial}{\partial t}F=\dot{F}^{ij}\nabla_i\nabla_jF+F\dot{F}^{ij}((h^2)_{ij}+cg_{ij}),
\end{equation}
where $(h^2)_{ij}=h_i^rh_{rj}$.
For any smooth symmetric function $G=G(h_i^j)$ of the Weingarten tensor, we have (see \cite{Andrews-Chen2017})
\begin{align}\label{2.2}
   \frac{\partial}{\partial t}G=& \dot{F}^{ij}\nabla_i\nabla_jG+\left(\dot{G}^{ij}\ddot{F}^{kl,mn}-\dot{F}^{ij}\ddot{G}^{kl,mn}\right)\nabla_ih_{kl}\nabla_jh_{mn} \nonumber\\
   & \quad+(F-\dot{F}^{kl}h_{kl})\dot{G}^{ij}(h^2)_{ij}+\dot{F}^{kl}(h^2)_{kl}\dot{G}^{ij}h_{ij}\nonumber\\
   &\quad +c\left( (F+\dot{F}^{kl}h_{kl})\dot{G}^{ij}g_{ij}-\dot{F}^{kl}g_{kl}\dot{G}^{ij}h_{ij}\right).
\end{align}

The key step in the proof of our results is to obtain the curvature pinching estimate of the flow, i.e., we prove that the ratio of the largest principal curvature $\k_1$ to the smallest one $\k_2$ is controlled by its initial value. In hyperbolic space case, we will prove this estimate by applying maximum principle to the evolution equation of the function $G$ defined in  \eqref{1.2}, which is obviously a smooth symmetric function $G=G(\mathcal{W})$ of the Weingarten tensor $\mathcal{W}$:
\begin{equation*}
  G=G(\mathcal{W})=\frac{|h|^2-2K}{(K-1)^2}F^2.
\end{equation*}
As before we equivalently write $G=g(\kappa)$, where $g$ is a smooth symmetric function of the principal curvatures.  The following proposition shows that the function $G$ (defined in \eqref{1.2}) has the advantage that its evolution equation \eqref{2.2} along the flow \eqref{1.1} has no zero-order terms, provided that the speed function $F$ is homogeneous of the principal curvatures, or the powers of the scalar curvature.
\begin{pro}\label{pro-2.1}
Let $\a>0$. Let $M_t$ be a smooth solution to the flow \eqref{1.1} with positive scalar curvature in $\H^3$. Assume either (i). the speed function $F$ is a homogeneous function of the principal curvatures, or (ii). $F=(K-1)^{\alpha}$ is a power of the scalar curvature. Then the evolution of $G$ defined in \eqref{1.2} satisfies
\begin{align}\label{2.4}
\frac{\partial}{\partial t}G=& \dot{F}^{ij}\nabla_i \nabla_j G + \(\dot{G}^{ij}\ddot{F}^{kl,mn}-\dot{F}^{ij}\ddot{G}^{kl,mn}\)\nabla_i h_{kl} \nabla_j h_{mn}.
\end{align}
\end{pro}
\begin{proof}
Since the function $G$ in \eqref{1.2} is a smooth symmetric function of $\mathcal{W}$, we have the evolution equation \eqref{2.2} for $G$ along the flow \eqref{1.1}. Write
\begin{equation*}
 F=f(\k),\qquad G=g(\k)=\frac{(\k_1-\k_2)^2}{(\k_1\k_2-1)^2}f^2(\k).
\end{equation*}
We have
\begin{align}
\dot{g}^1=&\frac{2(\k_1-\k_2)f}{\k_1\k_2-1}\left( \frac{(\k_1-\k_2)\dot{f}^1 }{\k_1\k_2-1}+\frac{(\k_2^2-1)f}{(\k_1\k_2-1)^2}\right),\label{2.9}\\
\dot{g}^2=&\frac{2(\k_1-\k_2)f}{\k_1\k_2-1}\left( \frac{(\k_1-\k_2)\dot{f}^2 }{\k_1\k_2-1}+\frac{(1-\k_1^2)f}{(\k_1\k_2-1)^2}\right).\label{2.9-b}
\end{align}
A direct calculation gives
\begin{align}
\dot{g}^1+\dot{g}^2=&2\frac{g}{f}(\dot{f}^1+\dot{f}^2)-\frac{2g(\k_1+\k_2)}{\k_1\k_2-1},\label{2.5}\\
\dot{g}^1\k_1+\dot{g}^2\k_2=&2(\a+1)g-\frac{4g\k_1\k_2}{\k_1\k_2-1},\\
\dot{g}^1\k_1^2+\dot{g}^2\k_2^2=&2\frac{g}{f}(\dot{f}^1\k_1^2+\dot{f}^2\k_2^2)-\frac{2g(\k_1+\k_2)}{\k_1\k_2-1}.\label{2.5-b}
\end{align}

(i). If $F$ is a homogeneous of degree $\a$ function of the principal curvatures, we have the Euler relation $\dot{f}^1\k_1+\dot{f}^2\k_2=\a f$. Using \eqref{2.F-1st}, the zero-order term of \eqref{2.2} for $G$ can be computed as follows:
\begin{align}\label{2.6}
Q_0
   =&(1-\a)(\dot{g}^{1}\k_1^2+\dot{g}^{2}\k_2^2)f+(\dot{g}^1\k_1+\dot{g}^2\k_2)(\dot{f}^1\k_1^2+\dot{f}^2\k_2^2)\nonumber\\
&\quad -(\a+1)(\dot{g}^{1}+\dot{g}^2)f+(\dot{g}^{1}\k_1+\dot{g}^2\k_2)(\dot{f}^1+\dot{f}^2).
\end{align}
Substituting the equations \eqref{2.5} -- \eqref{2.5-b} into \eqref{2.6}, we get $Q_0=0$.

(ii). If $F=(K-1)^{\alpha}$, the first order derivatives of $F$ and $G$ are given by:
\begin{align*}
  \dot{f}^1 =& \alpha(K-1)^{\alpha-1}\kappa_2,\quad  \dot{f}^2 =\alpha(K-1)^{\alpha-1}\kappa_1\\
  \dot{g}^1= & 2(\kappa_1-\kappa_2)(K-1)^{2(\alpha-1)}+2(\alpha-1)(\kappa_1-\kappa_2)^2(K-1)^{2\alpha-3}\kappa_2\\
  \dot{g}^2= & -2(\kappa_1-\kappa_2)(K-1)^{2(\alpha-1)}+2(\alpha-1)(\kappa_1-\kappa_2)^2(K-1)^{2\alpha-3}\kappa_1.
\end{align*}
Substituting these equations into the zero order term $Q_0$ of \eqref{2.2}, we also have $Q_0=0$.
\end{proof}

In the sphere case, we will prove the curvature pinching estimate by choosing the test function $G$ as defined in \eqref{s1:G1}. The evolution equation for $G$ has both the zero order terms and gradient terms. Since the function $G$  in \eqref{s1:G1} is the same one as in Euclidean case used in \cite{Andrews-Chen2012,Schulze2006}, the analysis of the gradient terms would be similar as the Euclidean case. We will describe this in details in \S \ref{sec:S}.

\subsection{Computation of the gradient terms}\label{sec:2-3}
Now we calculate the gradient term of \eqref{2.4} explicitly. Suppose $p$ is a point in $M$ where a new spatial maximum of $G$ is attained at time $t\in [0,T)$. Choose local orthonormal coordinates for $M$ near $p$ such that $h_{ij}(p,t)=\mrm{diag}(\k_1,\k_2)$. By \eqref{2.F-1st} and \eqref{2.F-2nd}, the gradient term on the RHS of \eqref{2.4} can be expressed as follows:
\begin{align}\label{2.7}
Q_1=&\(\dot{G}^{ij}\ddot{F}^{kl,mn}-\dot{F}^{ij}\ddot{G}^{kl,mn}\)\nabla_i h_{kl}\nabla_j h_{mn}\nonumber\\
=&\(\dot{g}^1\ddot{f}^{11}-\dot{f}^{1}\ddot{g}^{11}\)(\nabla_1 h_{11})^2+\(\dot{g}^1\ddot{f}^{22}-\dot{f}^{1}\ddot{g}^{22}\)(\nabla_1 h_{22})^2 \nonumber\\
&+2\(\dot{g}^1\ddot{f}^{12}-\dot{f}^{1}\ddot{g}^{12}\)\nabla_1 h_{11}\nabla_1 h_{22} \nonumber\\
&+\(\dot{g}^{2}\ddot{f}^{11}-\dot{f}^{2}\ddot{g}^{11}\)(\nabla_2 h_{11})^2+\(\dot{g}^{2}\ddot{f}^{22}-\dot{f}^{2}\ddot{g}^{22}\)(\nabla_2 h_{22})^2 \nonumber\\
&+2\(\dot{g}^{2}\ddot{f}^{12}-\dot{f}^{2}\ddot{g}^{12}\)\nabla_2 h_{11}\nabla_2 h_{22}\nonumber\\
&+2\frac{\dot{g}^1\dot{f}^{2}-\dot{g}^{2}\dot{f}^{1}}{\k_2-\k_1}(\nabla_1 h_{12})^2+2\frac{\dot{g}^1\dot{f}^{2}-\dot{g}^{2}\dot{f}^{1}}{\k_2-\k_1}(\nabla_2 h_{12})^2.
\end{align}
Without loss of generality, we assume that  $G$ is nonzero at $(p,t)$ (otherwise $M_t$ is a sphere and the proof is trivial). We also assume that $\k_1>\k_2$ because the maximum point of $G$ is not umbilical and both $F$ and $G$ are smooth and symmetric.

At the spatial maximum point of $G$, the gradient conditions $\nabla_i G=0$ give two equations:
\begin{align}\label{2.8}
\dot{g}^1\nabla_1 h_{11}+\dot{g}^2\nabla_1 h_{22}=0, \quad \dot{g}^1\nabla_2 h_{11}+\dot{g}^2\nabla_2 h_{22}=0.
\end{align}
For simplicity, we denote
\begin{align}
\b=&(\k_1-\k_2)(\k_1\k_2-1)\dot{f}^1+(\k_2^2-1)f,\label{2.10}\\
\g=&(\k_1-\k_2)(\k_1\k_2-1)\dot{f}^2-(\k_1^2-1)f.\label{2.10-b}
\end{align}
Then
\begin{equation}\label{s2:2-17}
  \dot{g}^1=\frac{2(\kappa_1-\kappa_2)f}{(\kappa_1\kappa_2-1)^3}\beta,\qquad \dot{g}^2=\frac{2(\kappa_1-\kappa_2)f}{(\kappa_1\kappa_2-1)^3}\gamma.
\end{equation}
Assume that at least one of $\dot{g}^1$ and $\dot{g}^2$ does not vanish. As $M_t$ is a family of surfaces with positive scalar curvature, i.e., $\k_1\k_2>1$, we have
\begin{enumerate}[(i)]
\item If $\dot{g}^1 \neq 0$, then $\b\neq 0$ and $\dfrac{\g}{\b}=\dfrac{\dot{g}^2}{\dot{g}^1}$;
\item If $\dot{g}^2 \neq 0$, then $\g\neq 0$ and $\dfrac{\b}{\g}=\dfrac{\dot{g}^1}{\dot{g}^2}$.
\end{enumerate}
Now we define
\begin{equation}\label{2.11}
\begin{split}
T_1^2:=\left\{
\begin{aligned}&\frac{(\nabla_1 h_{11})^2}{\g^2}, \quad\text{if $\g \neq 0$};\\
&\frac{(\nabla_1 h_{22})^2}{\b^2}, \quad\text{if $\b \neq 0$};
\end{aligned}\right. \quad \quad
T_2^2:=\left\{
\begin{aligned}&\frac{(\nabla_1 h_{22})^2}{\g^2}, \quad\text{if $\g \neq 0$};\\
&\frac{(\nabla_2 h_{22})^2}{\b^2}, \quad\text{if $\b \neq 0$};
\end{aligned}\right.
\end{split}
\end{equation}
Substituting \eqref{2.8} into \eqref{2.7} and using Codazzi equation, with the notation \eqref{2.11} we obtain
\begin{align}\label{2.13}
Q_1=& \biggl( (\dot{g}^1\ddot{f}^{11}-\dot{f}^1\ddot{g}^{11})\g^2-2(\dot{g}^1\ddot{f}^{12}-\dot{f}^1\ddot{g}^{12})\b\g \nonumber\\
 &\quad +\big(\dot{g}^1\ddot{f}^{22}-\dot{f}^1\ddot{g}^{22}+2\frac{\dot{g}^1\dot{f}^2-\dot{g}^2\dot{f}^1}{\k_2-\k_1}\big)\b^2\biggr)T_1^2\nonumber\\
 &+\biggl( (\dot{g}^2\ddot{f}^{22}-\dot{f}^2\ddot{g}^{22})\b^2 -2(\dot{g}^2\ddot{f}^{12}-\dot{f}^2\ddot{g}^{12})\b\g \nonumber\\
 &\quad +\big(\dot{g}^2\ddot{f}^{11}-\dot{f}^2\ddot{g}^{11}+2\frac{\dot{g}^1\dot{f}^2-\dot{g}^2\dot{f}^1}{\k_2-\k_1}\big)\g^2 \biggr)T_2^2.
\end{align}
We only calculate the coefficient (denoted by $\mathcal{Z}$) of $T_1^2$, since the coefficient of $T_2^2$ is similar, just with $\k_1$ and $\k_2$ interchanged.
Taking the further derivatives to \eqref{2.9} and \eqref{2.9-b} 
and by a direct calculation,  we get
\begin{align}\label{2.16}
\frac{\dot{g}^1\dot{f}^2-\dot{g}^2\dot{f}^1}{\k_2-\k_1}=-\frac{2 f^2\left[ (\k_2^2-1)\dot{f}^2+(\k_1^2-1)\dot{f}^1 \right]}{(\k_1\k_2-1)^3},
\end{align}
and
\begin{align}
\dot{g}^1\ddot{f}^{11}-\dot{f}^1\ddot{g}^{11}=&-\frac{8(\k_1-\k_2)(\k_2^2-1)f(\dot{f}^1)^2}{(\k_1\k_2-1)^3}-\frac{2(\k_1-\k_2)^2(\dot{f}^1)^3}{(\k_1\k_2-1)^2}\nonumber\\
                                              &-\frac{2(\k_2^2-1)f^2\dot{f}^1}{(\k_1\k_2-1)^3}+\frac{6\k_2(\k_1-\k_2)(\k_2^2-1)f^2\dot{f}^1}{(\k_1\k_2-1)^4} \nonumber \\
                                              &+\frac{2(\k_1-\k_2)(\k_2^2-1)f^2\ddot{f}^{11}}{(\k_1\k_2-1)^3},\label{2.17-a}\\
\dot{g}^1\ddot{f}^{22}-\dot{f}^1\ddot{g}^{22}=&-\frac{8(\k_2-\k_1)(\k_1^2-1)f\dot{f}^1\dot{f}^2}{(\k_1\k_2-1)^3}-\frac{2(\k_1-\k_2)^2\dot{f}^1(\dot{f}^2)^2}{(\k_1\k_2-1)^2}\nonumber\\
                                              &-\frac{2(\k_1^2-1)f^2\dot{f}^1}{(\k_1\k_2-1)^3}+\frac{6\k_1(\k_2-\k_1)(\k_1^2-1)f^2\dot{f}^1}{(\k_1\k_2-1)^4}\nonumber\\
                                              &+\frac{2(\k_1-\k_2)(\k_2^2-1)f^2\ddot{f}^{22}}{(\k_1\k_2-1)^3},\label{2.17-b}\\
\dot{g}^1\ddot{f}^{12}-\dot{f}^1\ddot{g}^{12}=&\frac{4(\k_1-\k_2)f\dot{f}^1}{(\k_1\k_2-1)^3}\left(\dot{f}^1(\k_1^2-1)-\dot{f}^2(\k_2^2-1)\right)+\frac{2f^2\dot{f}^1}{(\k_1\k_2-1)^2}\nonumber\\
                                              &-\frac{6(\k_1-\k_2)^2f^2 \dot{f}^1}{(\k_1\k_2-1)^4}+\frac{2(\k_2^2-1)(\k_1-\k_2)f^2\ddot{f}^{12}}{(\k_1\k_2-1)^3}.\label{2.17-c}
\end{align}
Then it follows from \eqref{2.16} -- \eqref{2.17-c} that
\begin{align}\label{2.19}
\mathcal{Z}=&\frac{2(\k_1-\k_2)(\k_2^2-1)f^2}{(\k_1\k_2-1)^3}\(\ddot{f}^{11}\g^2+\ddot{f}^{22}\b^2 -2\ddot{f}^{12}\g \b \) \nonumber\\
&+\frac{4f^2\( (\k_1^2-1)\dot{f}^1+ (\k_2^2-1)\dot{f}^2\)}{\k_1\k_2-1}\biggl( -2(\k_1-\k_2)f \dot{f}^1\nonumber\\
&\quad\quad -\frac{(\k_2^2-1)^2f^2}{(\k_1\k_2-1)^2} +(\k_1-\k_2)^2\dot{f}^1\dot{f}^2\biggr)\nonumber\\
&+\frac{4(\k_1-\k_2)^2f^2 \dot{f}^1}{\k_1\k_2-1}\biggl( -\frac{2(\k_1-\k_2)(\k_2^2-1)f\dot{f}^2}{\k_1\k_2-1} \nonumber\\
& \quad\quad +\frac{(\k_1^2-1)(\k_2^2-1)f^2}{(\k_1\k_2-1)^2} -(\k_1-\k_2)^2\dot{f}^1\dot{f}^2 \biggr).
\end{align}


\section{Flow in $\mathbb{H}^3$ by powers of mean curvature}\label{sec:3}
In this section, we study the flow by powers of mean curvature in $\H^3$, i.e.,
\begin{align}\label{3.1}
\left\{\begin{aligned}\frac{\partial}{\partial t} X(x,t)=&-H^\a(x,t) \nu(x,t),\quad \alpha>0\\
X(x,0)=&X_0(x).
\end{aligned}\right.
\end{align}

We first prove that $K>1$ is preserved along the flow (\ref{3.1}) in $\H^3$.
\begin{pro}\label{pro-3.1}
Let $M_t$, $t\in [0,T)$ be a smooth solution to the flow \eqref{3.1} in $\H^3$. For any power $\a>0$, if the Gauss curvature $K>1$ at $t=0$, then $K>1$ on $M_t$ for all $t\in [0,T)$. Moreover, we have
\begin{equation}\label{s3:K-0}
  K-1~\geq \min_{M_0}(K-1)\left(1-2^{\alpha}(\alpha+1)t \min_{M_0}(K-1)^{\frac{\alpha+1}2}\right)^{-\frac 2{\alpha+1}}
\end{equation}
on $M_t$ for $t\in [0,T)$. This implies an upper bound for the maximum existence time $T$:
\begin{equation*}
  T\leq \frac{2^{-\alpha}}{\alpha+1}\min_{M_0}(K-1)^{-\frac{\alpha+1}2}.
\end{equation*}
\end{pro}
\begin{proof}
Let $F=H^\a=f(\k)=(\k_1+\k_2)^\a$. Then
\begin{align*}
\dot{f}^1=\dot{f}^2=\a H^{\a-1},\quad \ddot{f}^{11}=\ddot{f}^{12}=\ddot{f}^{22}=\a(\a-1)H^{\a-2}.
\end{align*}
We apply \eqref{2.2} to calculate the evolution equation of the Gauss curvature $K$ along the flow \eqref{3.1}.  Let $G=K$ in \eqref{2.2}. We have
\begin{equation}\label{3.4}
\frac{\partial}{\partial t}G=\dot{F}^{ij}\nabla_i \nabla_j G + Q_1+Q_0,
\end{equation}
where we denote $Q_1,Q_0$ the gradient term and zero order term in the evolution of $G$. Since $G=g(\k)=\k_1\k_2$, the derivatives of $G$ are given by
\begin{align*}
\dot{g}^1=\k_2, \quad \dot{g}^2=\k_1, \quad \ddot{g}^{11}=\ddot{g}^{22}=0, \quad \ddot{g}^{12}=\ddot{g}^{21}=1.
\end{align*}
Then
a direct calculation gives that the zero order term $Q_0$ satisfies
\begin{align}\label{s3:1}
Q_0 = &(1-\a)K H^{\a+1}+2K(|A|^2+2)\a H^{\a-1}-(\a+1)H^{\a+1}\nonumber\\
=& (K-1)\left( H^2+\a(\k_1-\k_2)^2 \right) H^{\a-1}.
\end{align}
At the spatial minimum point of $K$, we have $\nabla_1K=\nabla_2K=0$, which implies that
\begin{equation*}
  \kappa_2\nabla_ih_{11}+\kappa_1\nabla_ih_{22}=0,\quad i=1,2.
\end{equation*}
By the general formula \eqref{2.7}, the gradient term $Q_1$ at the spatial minimum point of $K$ satisfies
\begin{align}\label{s3:2}
Q_1
=&\a\k_2\k_1^{-2}H^{\a-2}\left((\a-1)(\k_1-\k_2)^2+2H^2\right) (\nabla_1 h_{11})^2 \nonumber\\
&\quad +\a{\k_1}\k_2^{-2}H^{\a-2}\left((\a-1)(\k_1-\k_2)^2+2H^2\right) (\nabla_2 h_{22})^2
\end{align}
Applying maximum principle to \eqref{3.4} and using \eqref{s3:1} and \eqref{s3:2}, we have
\begin{equation}\label{s3:K-1}
  \frac d{dt}\min_{M_t}(K-1)~\geq ~(K-1)\left( H^2+\a(\k_1-\k_2)^2 \right) H^{\a-1},
\end{equation}
which implies that $K>1$ is preserved along the flow \eqref{3.1}. By the arithmetic-geometric means inequality $H=\kappa_1+\kappa_2\geq 2\sqrt{K}$ and $\alpha>0$, the inequality \eqref{s3:K-1} also implies
\begin{equation}\label{s3:K-2}
  \frac d{dt}\min_{M_t}(K-1)~\geq ~2^{\alpha+1}(K-1)^{1+\frac{\alpha+1}2}.
\end{equation}
The estimate \eqref{s3:K-0} follows by integrating \eqref{s3:K-2} in time.
\end{proof}

To show the curvature pinching estimate of $M_t$, we consider the following quantity on $M_t$:
\begin{align}\label{3.7}
G(x,t)=\frac{(\k_1+\k_2)^{2\a}(\k_1-\k_2)^2}{(\k_1\k_2-1)^2}.
\end{align}
Since $K=\kappa_1\kappa_2>1$ on $M_t$ for all $t\in [0,T)$, the quantity $G$ is well-defined. Firstly, we deduce the evolution equation of $G$ along the $H^\a$-flow \eqref{3.1}.
\begin{lem}\label{lemma-evol}
The evolution equation of $G$ along the $H^\a$-flow \eqref{3.1} satisfies
\begin{align}\label{3.8}
\frac{\partial}{\partial t}G=&\a H^{\a-1}\D G+\frac{2\a H^{5\a-3}}{(\k_1\k_2-1)^3}(a_1 T_1^2+a_2 T_2^2)
\end{align}
at the spatial maximum point, where $\Delta$ denotes the Laplacian operator with respect to the metric $g(t)$ on $M_t$, $T_1$ and $T_2$ are defined as \eqref{2.11}, and the coefficients $a_1$, $a_2$ are given by
\begin{align}\label{3.9}
a_1(\a,\k_1,\k_2):=&~4(\k_1-\k_2)^2(\k_1\k_2-1)^3 \cdot \a^2 \nonumber\\
                   &\quad-(\k_1-\k_2)(\k_1+\k_2)\biggl((3\k_2^2+1)\k_1^4+4\k_2(\k_2^2-3)\k_1^3 \nonumber \\
                   & \qquad+2(-3\k_2^4+\k_2^2+2)\k_1^2+4\k_2(\k_2^4-\k_2^2+2)\k_1 \nonumber\\
                   & \qquad +(-\k_2^6+\k_2^4-4)\biggr) \cdot \a \\
                   &\quad+(\k_1+\k_2) (\k_2^2-1)\biggl(\k_1^5-\k_2\k_1^4+(4-6\k_2^2)\k_1^3 \nonumber\\
                   & \qquad+2\k_2^3\k_1^2+(-3\k_2^4+12\k_2^2-8)\k_1-\k_2^5\biggr)\nonumber
\end{align}
and $a_2(\a,\k_1,\k_2):=a_1(\a,\k_2,\k_1)$.
\end{lem}
\begin{proof}
Firstly, since $F=H^{\alpha}$, the second order term in \eqref{2.4} becomes $\dot{F}^{ij}\nabla_i\nabla_jG=\alpha H^{\alpha-1}\Delta G$. This is just the first term of \eqref{3.8}. We next compute the gradient terms. Note that $\dot{g}^1,\dot{g}^2$ can not vanish at the same time. Using the equations \eqref{2.10} - \eqref{s2:2-17} and substituting $F=H^{\alpha}$, we have
\begin{equation*}
  (\dot{g}^1)^2+(\dot{g}^2)^2=\frac{2(\k_1-\k_2)^2H^{2\alpha}}{(\k_1\k_2-1)^6}(\b^2+\g^2),
\end{equation*}
where
\begin{align*}
\b=&\a H^{\a-1}(\k_1\k_2-1)(\k_1-\k_2)+H^{\a}(\k_2^2-1),\\
\g=&\a H^{\a-1}(\k_1\k_2-1)(\k_1-\k_2)-H^{\a}(\k_1^2-1).
\end{align*}
If $\beta=\gamma=0$, then $\k_1^2+\k_2^2-2=0$, which contradicts with
$$
\k_1^2+\k_2^2-2\geq 2(\k_1\k_2-1)>0.
$$
Therefore $\dot{g}^1,\dot{g}^2$ can not vanish at the same time.

By the general formula \eqref{2.19}, the coefficient in front of $T_1^2$ is given by
\begin{align*}
\mathcal{Z}
=&~\frac{2\a(\a-1)H^{5\a-2}(\k_1-\k_2)(\k_2^2-1)(\k_1^2+\k_2^2-2)^2}{(\k_1\k_2-1)^3} \\
 &+\frac{4\a H^{5\a-3}(\k_1^2+\k_2^2-2)}{\k_1\k_2-1}\biggl( -2(\k_1-\k_2)\a H \\
 &\quad \quad -\frac{(\k_2^2-1)^2 H^2}{(\k_1\k_2-1)^2} +(\k_1-\k_2)^2\a^2\biggr)\\
 &+\frac{4\a H^{5\a-3}(\k_1-\k_2)^2 }{\k_1\k_2-1}\biggl(-\frac{2(\k_1-\k_2)(\k_2^2-1)\a H}{\k_1\k_2-1} \\
 &\quad \quad +\frac{(\k_1^2-1)(\k_2^2-1)H^{2}}{(\k_1\k_2-1)^2} -(\k_1-\k_2)^2 \a^2  \biggr).
\end{align*}
Finally, we obtain
\begin{align*}
a_1:=&~\frac{(\k_1\k_2-1)^3}{2\a H^{5\a-3}}\mathcal{Z}\\
=&~(\a-1)(\k_1-\k_2)(\k_2^2-1)(\k_1^2+\k_2^2-2)^2H \\
&\quad-4\a H(\k_1-\k_2)(\k_1\k_2-1)^2(\k_1^2+\k_2^2-2)\\
&\quad-4\a H(\k_1-\k_2)^3(\k_2^2-1)(\k_1\k_2-1) \\
&\quad-2(\k_2^2-1)^2(\k_1^2+\k_2^2-2) H^2\\
&\quad+2(\k_1-\k_2)^2(\k_1^2-1)(\k_2^2-1)H^2 \\
&\quad +4\a^2(\k_1-\k_2)^2(\k_1\k_2-1)^3.
\end{align*}
Rewrite it as a quadratic polynomial of $\a$, we obtain the desired expression \eqref{3.9}.
\end{proof}

Now we apply maximum principle to the equation \eqref{3.8} to prove the monotonicity of $G$ for suitable range of $\alpha$.
\begin{theo}\label{theo-3.3}
Let $M_t$ be a smooth solution of \eqref{3.1} in $\H^3$ with $K>1$, where $\a\in [1/3,4]$. Then $\max_{M_t}G(\cdot,t)$ is non-increasing in time.
\end{theo}
\begin{proof}
To this theorem, we need to show the coefficients $a_1,a_2$ of the gradient terms are non-positive at the spatial maximum point of $G$. At the spatial maximum point $(p,t)$ of $G$, we assume that $G$ is nonzero (otherwise $M_t$ is a sphere and the proof is trivial).  Without loss of generality, we further assume that $\k_1>\k_2$. Define a convex subset $\mathcal{C}$ of $\R^2$ by
$$
\mathcal{C}:=\{ (\k_1,\k_2)\in \R^2 ~|~ \k_1\k_2>1, \k_1>\k_2 \}.
$$
By  \eqref{3.9}, the coefficients $a_i(\a,\k_1,\k_2)$, $i=1,2$ are strictly convex functions of $\a$ for any fixed point $(\k_1,\k_2)\in \mathcal{C}$. We claim that for all $(\k_1,\k_2)\in \mathcal{C}$,
$$
a_i(1/3,\k_1,\k_2) \leq 0, \quad a_i(4,\k_1,\k_2) \leq 0.
$$
This would imply that $a_i(\a,\k_1,\k_2) \leq 0$ for all $(\k_1,\k_2)\in \mathcal{C}$ provided that $\a\in [1/3,4]$.

To verify the claim, we denote $h_{i,\a}(x):=a_i(\a,x,\k_2)$, where $(x,\k_2)\in \mathcal{C}$. It suffices to show that
$$
h_{i,1/3}(\k_1)\leq 0, \quad h_{i,4}(\k_1)\leq 0.
$$
Note that for any $(\k_1,\k_2)\in \mathcal{C}$, we have $\k_1>\max\{\k_2,1/\k_2\}$.

\textbf{Claim 1:} $h_{i,1/3}(\k_1)\leq 0$, $i=1,2$.

We denote $h_{i,\a}^{(k)}(x)$ the $k$-th derivative of $h_{i,\a}(x)$ with respect to $x$. Clearly, we have
$$
h_{1,1/3}^{(6)}(x)=-960<0, \quad h_{2,1/3}^{(7)}(x)=-13440(\k_2+4x)<0,
$$
for all $x>\k_2>0$.
\begin{enumerate}[(i)]
\item If $\k_2\geq \frac{1}{\k_2}$, then $\k_2\geq 1$; The derivatives of $h_{i,1/3}(x)$ at $x=\k_2$ satisfy
\begin{align*}
h_{1,1/3}^{(5)}(\k_2)=&-\frac{160}{3}\k_2(9+2\k_2^2) \leq 0;\\
h_{1,1/3}^{(4)}(\k_2)=&-32(4-7\k_2^2+7\k_2^4) \leq 0;\\
h_{1,1/3}^{(3)}(\k_2)=&-32\k_2(\k_2^2-1)(6\k_2^2-5) \leq 0;\\
h_{1,1/3}^{(2)}(\k_2)=&-\frac{32}{9}(\k_2^2-1)^2(29\k_2^2-1)\leq 0;\\
h_{1,1/3}^{(1)}(\k_2)=&-\frac{16}{3}\k_2(\k_2^2-1)^2(8\k_2^2-5)\leq 0;\\
h_{1,1/3}(\k_2)=&-16\k_2^2(\k_2^2-1)^3\leq 0.
\end{align*}
and
\begin{align*}
h_{2,1/3}^{(6)}(\k_2)=&-960(-1+44\k_2^2)\leq 0;\\
h_{2,1/3}^{(5)}(\k_2)=&-\frac{160}{3}\k_2(-51+338\k_2^2)\leq 0;\\
h_{2,1/3}^{(4)}(\k_2)=&-32\k_2^2(-79+187\k_2^2)\leq 0;\\
h_{2,1/3}^{(3)}(\k_2)=&-32\k_2(-3+4\k_2^2)(-1+13\k_2^2)\leq 0;\\
h_{2,1/3}^{(2)}(\k_2)=&-\frac{32}{9}(\k_2^2-1)(-1-40\k_2^2+113\k_2^4) \leq 0;\\
h_{2,1/3}^{(1)}(\k_2)=&-\frac{16}{3}\k_2(\k_2^2-1)^2(16\k_2^2-1)\leq 0;\\
h_{2,1/3}(\k_2)=&-16\k_2^2(\k_2^2-1)^3\leq 0.
\end{align*}
\item If $\k_2\leq \frac{1}{\k_2}$, then $0<\k_2\leq 1$; The derivatives of $h_{i,1/3}(x)$ at $x=1/\k_2$ satisfy
\begin{align*}
h_{1,1/3}^{(5)}(1/\k_2)=&-\frac{160}{3\k_2}(18-9\k_2^2+2\k_2^4) \leq 0;\\
h_{1,1/3}^{(4)}(1/\k_2)=&-\frac{32}{3\k_2^2}(45-33\k_2^2-11\k_2^4+11\k_2^6) \leq 0;\\
h_{1,1/3}^{(3)}(1/\k_2)=&-\frac{16}{3\k_2^3}(1-\k_2^2)(5+4\k_2^2)(6-3\k_2^2-\k_2^4)\leq 0;\\
h_{1,1/3}^{(2)}(1/\k_2)=&-\frac{8}{3\k_2^4}(1-\k_2^2)^2(15+24\k_2^2+3\k_2^4+2\k_2^6)\leq 0;\\
h_{1,1/3}^{(1)}(1/\k_2)=&-\frac{4}{3\k_2^5}(1-\k_2^2)^2(6+13\k_2^2-6\k_2^4-3\k_2^6+2\k_2^8) \leq 0;\\
h_{1,1/3}(1/\k_2)=&-\frac{4}{3\k_2^6}(1-\k_2^2)^4(1+\k_2^2)(1+4\k_2^2+\k_2^4)\leq 0.
\end{align*}
and
\begin{align*}
h_{2,1/3}^{(6)}(1/\k_2)=&-\frac{960}{\k_2^2}(28+13\k_2^2+2\k_2^4)\leq 0;\\
h_{2,1/3}^{(5)}(1/\k_2)=&-\frac{160}{3\k_2^3}(168+108\k_2^2+3\k_2^4+8\k_2^6)\leq 0;\\
h_{2,1/3}^{(4)}(1/\k_2)=&-\frac{32}{3\k_2^4}(210+165\k_2^2-75\k_2^4+13\k_2^6+11\k_2^8)\leq0;\\
h_{2,1/3}^{(3)}(1/\k_2)=&-\frac{16}{3\k_2^5}(84+75\k_2^2-105\k_2^4+4\k_2^6+13\k_2^8+\k_2^{10})\leq0;\\
h_{2,1/3}^{(2)}(1/\k_2)=&-\frac{8}{3\k_2^6}(1-\k_2^2)(28+55\k_2^2-25\k_2^4-15\k_2^6-5\k_2^8)\leq 0;\\
h_{2,1/3}^{(1)}(1/\k_2)=&-\frac{4}{3\k_2^7}(1-\k_2^2)^2(8+24\k_2^2-3\k_2^4-14\k_2^6-3\k_2^8) \leq 0;\\
h_{2,1/3}(1/\k_2)=&-\frac{4}{3\k_2^8}(1-\k_2^2)^4(1+\k_2^2)(1+4\k_2^2+\k_2^4)\leq 0.
\end{align*}
\end{enumerate}
In both cases, we have $h_{i,1/3}(\k_1)\leq 0$.

\textbf{Claim 2:} $h_{i,4}(\k_1)\leq 0$, $i=1,2$.

Firstly, we have
$$
h_{1,4}^{(6)}(x)=-720(5+11\k_2^2)<0,\quad h_{2,4}^{(7)}(x)=-20160(10x-3\k_2)<0,
$$
for all $x>\k_2>0$.
\begin{enumerate}[(i)]
\item If $\k_2\geq \frac{1}{\k_2}$, then $\k_2\geq 1$; The derivatives of $h_{i,4}$ at $x=\k_2$:
\begin{align*}
h_{1,4}^{(5)}(\k_2)=&-2160 \k_2(\k_2^2-1) \leq 0;\\
h_{1,4}^{(4)}(\k_2)=&-96(5+5\k_2^2+6\k_2^4)\leq 0;\\
h_{1,4}^{(3)}(\k_2)=&-456\k_2(\k_2^2-1)(\k_2^2+1)\leq 0;\\
h_{1,4}^{(2)}(\k_2)=&-80(\k_2^2-1)^2(3\k_2^2+1) \leq 0;\\
h_{1,4}^{(1)}(\k_2)=&-8\k_2(\k_2^2-1)^2(9\k_2^2-7) \leq 0;\\
h_{1,4}(\k_2)=&-16\k_2^2(\k_2^2-1)^3\leq 0.
\end{align*}
and
\begin{align*}
h_{2,4}^{(6)}(\k_2)=&-720(-5+77\k_2^2) \leq 0;\\
h_{2,4}^{(5)}(\k_2)=&-240\k_2(-15+47\k_2^2) \leq 0;\\
h_{2,4}^{(4)}(\k_2)=&-96\k_2^2(25+11\k_2^2) \leq 0;\\
h_{2,4}^{(3)}(\k_2)=&-24\k_2(-51+60\k_2^2+7\k_2^4) \leq 0;\\
h_{2,4}^{(2)}(\k_2)=&-32(\k_2^2-1)(-5+9\k_2^2+\k_2^4) \leq 0;\\
h_{2,4}^{(1)}(\k_2)=&-8\k_2(\k_2^2-1)^2(3+7\k_2^2)\leq 0;\\
h_{2,4}(\k_2)=&-16\k_2^2(\k_2^2-1)^3\leq 0.
\end{align*}
\item If $\k_2\leq \frac{1}{\k_2}$, then $0<\k_2\leq 1$; The derivatives of $h_{i,4}$ at $x=1/\k_2$:
\begin{align*}
h_{1,4}^{(5)}(1/\k_2)=&-\frac{720}{\k_2}(1-\k_2^2)(5+8\k_2^2)\leq0;\\
h_{1,4}^{(4)}(1/\k_2)=&-\frac{24}{\k_2^2}(75-55\k_2^2-55\k_2^4+99\k_2^6)\leq 0;\\
h_{1,4}^{(3)}(1/\k_2)=&-\frac{24}{\k_2^3}(1-\k_2^2)(25-20\k_2^2+6\k_2^4+15\k_2^6)\leq 0;\\
h_{1,4}^{(2)}(1/\k_2)=&-\frac{2}{\k_2^4}(1-\k_2^2)^2(75-45\k_2^2+37\k_2^4+21\k_2^6)\leq 0;\\
h_{1,4}^{(1)}(1/\k_2)=&-\frac{2}{\k_2^5}(1-\k_2^2)^2(15-17\k_2^2+7\k_2^4+9\k_2^6-6\k_2^8) \leq 0;\\
h_{1,4}(1/\k_2)=&-\frac{1}{\k_2^6}(1-\k_2^2)^4(1+\k_2^2)(5-2\k_2^2+5\k_2^4)\leq 0.
\end{align*}
and
\begin{align*}
h_{2,4}^{(6)}(1/\k_2)=&-\frac{720}{\k_2^2}(140-89\k_2^2+21\k_2^4) \leq 0;\\
h_{2,4}^{(5)}(1/\k_2)=&-\frac{240}{\k_2^3}(140-141\k_2^2+63\k_2^4-30\k_2^6)\leq 0;\\
h_{2,4}^{(4)}(1/\k_2)=&-\frac{24}{\k_2^4}(350-495\k_2^2+315\k_2^4-125\k_2^6+99\k_2^8) \leq 0;\\
h_{2,4}^{(3)}(1/\k_2)=&-\frac{24}{\k_2^5}(70-130\k_2^2+105\k_2^4-26\k_2^6+9\k_2^8-12\k_2^{10}) \leq 0;\\
h_{2,4}^{(2)}(1/\k_2)=&-\frac{2}{\k_2^6}(1-\k_2^2)(140-187\k_2^2+128\k_2^4+46\k_2^6-52\k_2^8-11\k_2^{10}) \leq 0;\\
h_{2,4}^{(1)}(1/\k_2)=&-\frac{2}{\k_2^7}(1-\k_2^2)^2(20-17\k_2^2+9\k_2^4+9\k_2^6-13\k_2^8)\leq 0;\\
h_{2,4}(1/\k_2)=&-\frac{1}{\k_2^8}(1-\k_2^2)^4(1+\k_2^2)(5-2\k_2^2+5\k_2^4)\leq 0.
\end{align*}
\end{enumerate}
In both cases, we have $h_{i,4}(\k_1)\leq 0$.

In summary, we conclude that the gradient terms are non-positive at the spatial maximum point of $G$. Applying the maximum principle we conclude that $G$ is monotone non-increasing along the $H^\a$-flow with $\a \in [1/3,4]$, which completes the proof.
\end{proof}


We have proved that $\max_{M_t}G(x,t)$ is monotone non-increasing along the $H^\a$-flow for $\a\in [1/3,4]$. By assuming further $\a\in [1,4]$, we can obtain the curvature pinching estimate of $M_t$.
\begin{cor}\label{cor-3.5}
Let $M_t, t\in [0,T)$ be a smooth solution of $H^\a$-flow with $K>1$ in $\H^3$, where $\a\in [1,4]$. There exists a constant $C>0$ depending only on the initial surface $M_0$ and $\alpha$ such that
\begin{align}\label{s3:pinc-1}
0<\frac{1}{C}\leq \frac{\k_1}{\k_2} \leq C
\end{align}
on $M_t$ for all $t\in [0,T)$.
\end{cor}
\begin{proof}
By the monotonicity of $G$, we have
$$
\frac{(\k_1+\k_2)^{2\a}(\k_1-\k_2)^2}{(\k_1\k_2-1)^2}\leq C_1:=\max_{M_0}G(\cdot,0).
$$
If $\a\geq 1$, then
\begin{align*}
\frac{\k_1}{\k_2}+\frac{\k_2}{\k_1}-2= &\frac{(\k_1+\k_2)^{2\a}(\k_1-\k_2)^2}{(\k_1\k_2-1)^2}\frac{(\k_1\k_2-1)^2}{\k_1\k_2(\k_1+\k_2)^{2\alpha}}\\
\leq &2^{-2\alpha}C_1(\k_1\k_2-1)^{1-\alpha},
\end{align*}
which is bounded from above by Proposition \eqref{pro-3.1} and $\alpha\geq 1$. The estimate \eqref{s3:pinc-1} follows immediately.
\end{proof}

\section{Flow in $\mathbb{H}^3$ by powers of scalar curvature}\label{sec:4}
In this section, we consider the flow for surfaces in $\H^3$ by powers of scalar curvature, i.e.,
\begin{align}\label{4.1}
\left\{\begin{aligned}\frac{\partial}{\partial t} X(x,t)=&-(K(x,t)-1)^\a \nu(x,t),\quad \a>0\\
X(x,0)=&X_0(x).
\end{aligned}\right.
\end{align}

\begin{pro}\label{pro-4.1}
Let $M_t$, $t\in [0,T)$ be a smooth solution of \eqref{4.1} in $\H^3$. For any power $\a>0$, if $K>1$ on $M_0$, then $K>1$ on $M_t$ for all $t\in [0,T)$. Moreover, we have the estimate
\begin{equation}\label{s4:K1}
K-1\geq \min_{M_0}(K-1)\left(1-(2\alpha+1)\min_{M_0}(K-1)^{\alpha+\frac 12}t\right)^{-\frac 2{2\alpha+1}}
\end{equation}
on $M_t$ for $t\in [0,T)$.
\end{pro}
\begin{proof}
By \eqref{s2:speed}, the speed function $F=(K-1)^{\alpha}$ of the flow \eqref{4.1} satisfies the following evolution equation
\begin{align}\label{s4:K-evl1}
  \frac{\partial}{\partial t}(K-1)^{\alpha} =& \alpha (K-1)^{\alpha-1}\dot{K}^{ij}\nabla_i\nabla_j(K-1)^{\alpha} +\alpha(K-1)^{2\alpha-1}\dot{K}^{ij}\left((h^2)_{ij}-g_{ij}\right) \nonumber\\
  = & \alpha (K-1)^{\alpha-1}\dot{K}^{ij}\nabla_i\nabla_j(K-1)^{\alpha} +\alpha(K-1)^{2\alpha}H.
\end{align}
By applying the maximum principle to \eqref{s4:K-evl1}, the spatial minimum of $K-1$ is non-decreasing in time and remains positive for $t>0$. Then the equation \eqref{s4:K-evl1} is equivalent to
\begin{align*}
  \frac{\partial}{\partial t}(K-1) = & \dot{K}^{ij}\nabla_i\nabla_j(K-1)^{\alpha} +(K-1)^{\alpha+1}H,
\end{align*}
which implies
\begin{align}\label{s4:K-evl2}
  \frac{d}{d t}\min_{M_t}(K-1) \geq & 2(K-1)^{\alpha+\frac 32}.
\end{align}
Integrating \eqref{s4:K-evl2} gives the estimate \eqref{s4:K1}.
\end{proof}

Following the similar idea as before, we consider the quantity
\begin{equation}\label{s4:G}
G(x,t):=(\k_1\k_2-1)^{2\a-2}(\k_1-\k_2)^2
\end{equation}
to deduce the pinching estimate of the principal curvatures of $M_t$. Applying Proposition \ref{pro-2.1} and the general formula \eqref{2.19}, we have the evolution equation of $G$ along the flow \eqref{4.1}.
\begin{lem}\label{lem-4.2}
Along the flow \eqref{4.1}, the evolution of $G$ satisfies
\begin{align}\label{4.6}
\frac{\partial}{\partial t}G= \alpha (K-1)^{\alpha-1}\dot{K}^{ij}\nabla_i\nabla_j G+2\a(K-1)^{5\a-3}(a_1 T_1^2+a_2 T_2^2)
\end{align}
at the spatial maximum point of $G$. Here $T_1$, $T_2$ are defined as \eqref{2.11}, and the coefficients $a_1$, $a_2$ are given by
\begin{align}\label{4.7}
a_1(\a,\k_1,\k_2):=&4\k_1 \k_2^2 (\k_1-\k_2)^2 \cdot \a^2 \nonumber\\
                   &+(\k_1-\k_2)(\k_1^2 - 2 \k_1 \k_2 + 5 \k_2^2 - 5 \k_1^2 \k_2^2 + 2 \k_1 \k_2^3 - \k_2^4)\cdot \a \nonumber\\
                   &+(\k_2^2-1)(\k_1^3 + 4 \k_2 - 3 \k_1^2 \k_2 - \k_1 \k_2^2 - \k_2^3)
\end{align}
and $a_2(\a,\k_1,\k_2):=a_1(\a,\k_2,\k_1)$.
\end{lem}

We now apply maximum principle to \eqref{4.6} to prove the monotonicity of $\max_{M_t}G(x,t)$ along the flow \eqref{4.1} with $\a\in [1/4,1]$.
\begin{theo}\label{theo-4.3}
Let $M_t$, $t\in [0,T)$ be a smooth solution of the flow \eqref{4.1} in $\H^3$ with $K>1$, where $\a\in [1/4,1]$. Then $\max_{M_t}G(x,t)$ is monotone non-increasing in time.
\end{theo}
\begin{proof}
We need to show that the coefficients $a_i(\alpha,\k_1,\k_2)$ are non-positive at the spatial maximum point of $G$. Similar to the proof of Theorem \ref{theo-3.3}, we denote
$$
\mathcal{C}:=\{ (\k_1,\k_2)\in \R^2 ~|~ \k_1\k_2>1, \k_1>\k_2 \}.
$$
We assume $\k_1>\k_2$ without loss of generality. Then the expression \eqref{4.7} says that both the coefficients $a_i(\a,\k_1,\k_2)$, $i=1,2$ are strictly convex functions of $\a$ for each fixed $(\k_1,\k_2)\in \mathcal{C}$. We obviously have
\begin{align*}
a_1(1,\k_1,\k_2)=&~-4\k_2(\k_1\k_2-1)^2\leq 0,\\
a_2(1,\k_1,\k_2)=&~-4\k_1(\k_1\k_2-1)^2\leq 0.
\end{align*}
We claim that
$$
a_i(1/4,\k_1,\k_2) \leq 0,\qquad \forall~(\k_1,\k_2)\in \mathcal{C}.
$$
Therefore, $a_i(\a,\k_1,\k_2) \leq 0$ for all $(\k_1,\k_2)\in \mathcal{C}$ provided that $\a\in [1/4,1]$.

To verify the claim, we rewrite $a_1(1/4,\k_1,\k_2)$ as
\begin{align*}
a_1(1/4,\k_1,\k_2)=&-\frac{1}{4}(\k_1-\k_2)^2\left(3(\k_1-\k_2)+7\k_2^3 \right)\\
                  &-5\k_2^2(\k_2^2-1)(\k_1-\k_2)-4\k_2(\k_2^2-1)^2 \\
                  =&-\frac{1}{4}\left(3(\k_1-\k_2)+7\k_2^3 \right)  x^2 -5\k_2^2 xy -4\k_2 y^2,
\end{align*}
where $x=\k_1-\k_2$ and $y=\k_2^2-1$. For $(\k_1,\k_2)\in \mathcal{C}$, $a_1(1/4,\k_1,\k_2)$ is a strictly concave quadratic polynomial of $x,y$, with discriminant
$$
\D=25\k_2^4 -4\k_2\left(3(\k_1-\k_2)+7\k_2^3 \right)=-3\k_2^4-12\k_2(\k_1-\k_2)<0,
$$
which implies that $a_1(1/4,\k_1,\k_2)\leq 0$ for all $(\k_1,\k_2)\in \mathcal{C}$.

On the other hand, to show that $a_2(1/4,\k_1,\k_2)\leq 0$ for all $(\k_1,\k_2)\in \mathcal{C}$, we introduce the auxiliary function
\begin{align*}
h(x):=&a_2(1/4,x,\k_2)\\
=&\frac{1}{4} \biggl( -3 \k_2^3 + (-16 + 9 \k_2^2) x + 11 \k_2 x^2 + (15 - 7 \k_2^2) x^3 -
 6 \k_2 x^4 - 3 x^5\biggr)
\end{align*}
and denote $h^{(k)}(x)$ the $k$-th derivatives of $h(x)$. Then we have
\begin{align*}
h^{(2)}(x)=&-\frac{1}{2}\(-11\k_2-45 x +21\k_2^2+36\k_2 x^2+30x^3\) \\
=&-\frac{1}{2} \biggl( 11\k_2(x^2-1)+25x(\k_2 x-1)+20x(x^2-1)+10x^3+21\k_2^2 \biggr).
\end{align*}
Hence, we have $h^{(2)}(\k_1)\leq 0$ for all $\k_1>\max\{\k_2,1/\k_2\}\geq 1$.
\begin{enumerate}[(i)]
\item If $\k_2\geq \frac{1}{\k_2}$, then $\k_2 \geq 1$ and
\begin{align*}
h^{(1)}(\k_2)=&-(\k_2-1)(15\k_2^2-4)\leq 0;\\
h(\k_2)=&-4\k_2(\k_2-1)^2\leq 0;
\end{align*}
\item If $\k_2\leq \frac{1}{\k_2}$, then $0<\k_2 \leq 1$ and
\begin{align*}
h^{(1)}(1/\k_2)=&-\frac{3}{4\k_2^4}(1-\k_2^2)(5-2\k_2^2+3\k_2^4)\leq 0;\\
h(1/\k_2)=&-\frac{3}{4\k_2^5}(1-\k_2^2)^2(1-\k_2^2+\k_2^4)\leq 0;
\end{align*}
\end{enumerate}
In both cases, we have $h(\k_1)=a_2(1/4,\k_1,\k_2)\leq 0$ for any $(\k_1,\k_2)\in \mathcal{C}$.

Finally, by maximum principle we conclude that $\max_{M_t}G(x,t)$ is monotone non-increasing along the flow \eqref{4.1} with $\a \in [\frac{1}{4},1]$.
\end{proof}

\begin{cor}\label{s4:cor}
Let $M_t, t\in [0,T)$ be a smooth solution to the flow \eqref{4.1} with positive scalar curvature in $\H^3$, with $\a\in [1/2,1]$. Then there exists a constant $C>0$ depending only on $M_0$ and $\alpha$ such that
\begin{align}\label{s4:pinc}
0<\frac{1}{C}\leq \frac{\k_1}{\k_2} \leq C
\end{align}
on $M_t$ for all $t\in [0,T)$.
\end{cor}
\begin{proof}
As in the proof of Corollary \ref{cor-3.5},
\begin{align*}
\frac{\k_1}{\k_2}+\frac{\k_2}{\k_1}-2= &(\k_1\k_2-1)^{2\a-2}(\k_1-\k_2)^2\cdot \frac{1}{\k_1\k_2(\k_1\k_2-1)^{2\a-2}} \\
\leq & \max_{M_0}G(\cdot,0) \frac{1}{(\k_1\k_2-1)^{2\a-1}},
\end{align*}
which is bounded from above by Proposition \ref{pro-4.1} and $\alpha\geq 1/2$.  The pinching estimate \eqref{s4:pinc} follows immediately.
\end{proof}

\section{Flow in $\mathbb{H}^3$ by powers of Gauss curvature}\label{sec:5}
In this section, we study the flow of surfaces in $\mathbb{H}^3$ by powers of Gauss curvature, i.e.,
\begin{align}\label{5.1}
\left\{\begin{aligned}\frac{\partial}{\partial t} X(x,t)=&-K^\a(x,t) \nu(x,t),\qquad \alpha>0\\
X(x,0)=&X_0(x).
\end{aligned}\right.
\end{align}

\begin{pro}\label{pro-5.1}
Let $M_t$, $t\in [0,T)$ be a smooth solution to the flow \eqref{5.1} in $\H^3$. For any power $\a>0$, if $M_0$ has positive scalar curvature, then $M_t$ has positive scalar curvature for all $t\in [0,T)$. Moreover, we have the estimate
\begin{equation}\label{s5:K1}
K-1\geq \min_{M_0}(K-1)\left(1-(2\alpha+1)\min_{M_0}(K-1)^{\alpha+\frac 12}t\right)^{-\frac 2{2\alpha+1}}
\end{equation}
on $M_t$ for $t\in [0,T)$.
\end{pro}
\begin{proof}
By the equation \eqref{s2:speed}, the speed function $F=K^{\alpha}$ of the flow \eqref{5.1} evolves by
\begin{align}\label{5.2}
\frac{\partial}{\partial t}K^{\alpha}=&~\alpha K^{\alpha-1}\dot{K}^{ij}\nabla_i\nabla_jK^{\alpha}+\alpha K^{2\alpha-1}\dot{K}^{ij}((h^2)_{ij}-g_{ij})\nonumber\\
=&~\alpha K^{\alpha-1}\dot{K}^{ij}\nabla_i\nabla_jK^{\alpha}+\alpha K^{2\alpha-1}H(K-1).
\end{align}
Since $\alpha>0$, the maximum principle applied to \eqref{5.2} implies that $K>1$ is preserved along the flow \eqref{5.1}. The proof of the estimate \eqref{s5:K1} is similar as in Proposition \ref{pro-4.1}.
\end{proof}

We define
$$
G(x,t):=\frac{(\k_1\k_2)^{2\a}(\k_1-\k_2)^2}{(\k_1\k_2-1)^2}
$$
on $M_t$. By Proposition \ref{pro-5.1}, the function $G$ is well-defined on $M_t$ for all $t\in [0,T)$.
\begin{theo}\label{theo-5.3}
Let $M_t$, $t\in [0,T)$ be a smooth solution of the flow \eqref{5.1} in $\H^3$ with $K>1$, where $\a\in [1/4,1]$. Then $\max_{M_t}G(x,t)$ is monotone non-increasing in time.
\end{theo}
\begin{proof}
The evolution of $G$ along the flow \eqref{5.1} satisfies
\begin{align}\label{5.7}
\frac{\partial}{\partial t}G= \alpha K^{\alpha-1}\dot{K}^{ij}\nabla_i\nabla_jG+2\a\frac{(\k_1\k_2)^{5\a-2}}{(\k_1\k_2-1)^2}(a_1 T_1^2+a_2 T_2^2)
\end{align}
at the spatial maximum point of $G$. Here $T_1$, $T_2$ are defined as \eqref{2.11}, and the coefficients $a_1$, $a_2$ are given by
\begin{align*}
a_1(\a,\k_1,\k_2):=&4\k_2 (\k_1-\k_2)^2(\k_1\k_2-1)^2 \cdot \a^2 \nonumber\\
                   &+(\k_1-\k_2)(\k_1\k_2-1)\left[\k_1^2(1-5\k_2^2) - 2 \k_1(\k_2-\k_2^3) + (5\k_2^2-\k_2^4)\right]\cdot \a \nonumber\\
                   &+(\k_2^2-1)\left[\k_1^4 \k_2 +\k_1^3(1-3\k_2^2) + \k_1^2 (\k_2-\k_2^3) +  \k_1 (3\k_2^2-\k_2^4) -\k_2^3\right],
\end{align*}
and $a_2(\a,\k_1,\k_2):=a_1(\a,\k_2,\k_1)$. This follows from Proposition \ref{pro-2.1} and the formula \eqref{2.19}.

To prove Theorem \ref{theo-5.3}, we need to show $a_i$ are non-positive at the spatial maximum point of $G$. Both the coefficients $a_i(\a,\k_1,\k_2)$, $i=1,2$ are strictly convex functions of $\a$ in the cone $\mathcal{C}=\{ (\k_1,\k_2)\in \R^2 ~|~ \k_1\k_2>1, \k_1>\k_2 \}$.
We claim that for all $(\k_1,\k_2)\in \mathcal{C}$,
$$
a_i(1/4,\k_1,\k_2) \leq 0, \quad a_i(1,\k_1,\k_2) \leq 0.
$$
Then, for all $(\k_1,\k_2)\in \mathcal{C}$, $a_i(\a,\k_1,\k_2) \leq 0$ provided that $\a\in [1/4,1]$.

To verify the claim, we denote $g_{i,\a}(x):=a_i(\a,x,\k_2)$, where $(\k_1,\k_2)\in \mathcal{C}$. It suffices to show that
$$
g_{i,1/4}(\k_1)\leq 0, \quad g_{i,1}(\k_1)\leq 0.
$$
Firstly, we show that $g_{i,1}(\k_1)\leq 0$, $i=1,2$. We have
\begin{align*}
g_{1,1}^{(2)}(x)
                  =&-12 \left[(\k_2^2+1)(x-\k_2)+2\k_2^3(\k_2 x-1) \right],
\end{align*}
and
\begin{align*}
g_{2,1}^{(3)}(x)
                =&-12 \left[ 9(x^2-\k_2^2)+5(\k_2 x-1)+x^2 + 8 \k_2^3 x +7 \k_2 x  \right].
\end{align*}
Hence, we have $g_{1,1}^{(2)}(\k_1)\leq 0$ and $g_{2,1}^{(3)}(\k_1)\leq 0$ for all $\k_1>\max\{\k_2,1/\k_2\}\geq 1$.
\begin{enumerate}[(i)]
\item If $\k_2\geq \frac{1}{\k_2}$, then $\k_2 \geq 1$ and
\begin{align*}
g_{1,1}^{(1)}(\k_2)=&-12\k_2^2(\k^2_2-1)^2 \leq 0;\\
g_{1,1}(\k_2)=&-4\k_2^3(\k^2_2-1)^2\leq 0;
\end{align*}
and
\begin{align*}
g_{2,1}^{(2)}(\k_2)=&-8\k_2(-3+\k_2^2+6\k_2^4) \leq 0;\\
g_{2,1}^{(1)}(\k_2)=&-16\k_2^4(\k_2^2-1)\leq 0;\\
g_{2,1}(\k_2)=&-4\k_2^3(\k_2^2-1)^2\leq 0;
\end{align*}
\item If $\k_2\leq \frac{1}{\k_2}$, then $0<\k_2 \leq 1$ and
\begin{align*}
g_{1,1}^{(1)}(1/\k_2)=&-\frac{6}{\k_2^2}(\k_2^2-1)^2(1+\k_2^2)\leq 0;\\
g_{1,1}(1/\k_2)=&-\frac{2}{\k_2^3}(\k_2^2-1)^2(1+\k_2^4)\leq 0;
\end{align*}
and
\begin{align*}
g_{2,1}^{(2)}(1/\k_2)=&-\frac{4}{\k_2^3}(1+\k_2^2)(2-\k_2^2)(5-\k_2^2)\leq 0;\\
g_{2,1}^{(1)}(1/\k_2)=&-\frac{2}{\k_2^4}(1-\k_2^2)(5+2\k_2^2+\k_2^4)\leq 0;\\
g_{2,1}(1/\k_2)=&-\frac{2}{\k_2^5}(1-\k_2^2)^2(1+\k_2^4)\leq 0;
\end{align*}
\end{enumerate}
In both cases, we have $g_{i,1}(\k_1)\leq 0$.

Now we show that $g_{i,1/4}(\k_1)\leq 0$. We have
\begin{align*}
g_{1,1/4}^{(3)}(x)
                  =&-\frac{3}{2}\left[12\k_2(x-\k_2)+(2\k_2^2-1)^2+3\k_2^4+4 \right] \leq 0,
\end{align*}
and
\begin{align*}
g_{2,1/4}^{(3)}(x)=-3\left[ -5+14\k_2^3 x+25 x^2+6\k_2 x(-4+5x^2)+6\k_2^2(-1+5x^2) \right]\leq 0,
\end{align*}
Hence, we have $g_{1,1/4}^{(3)}(\k_1)\leq 0$ and $g_{2,1/4}^{(3)}(\k_1)\leq 0$ for all $\k_1>\max\{\k_2,1/\k_2\}\geq 1$.
\begin{enumerate}[(i)]
\item If $\k_2\geq \frac{1}{\k_2}$, then $\k_2 \geq 1$ and
\begin{align*}
g_{1,1/4}^{(2)}(\k_2)=&-\frac{3}{2}\k_2 (\k_2^2-1)(9\k_2^2-5)\leq 0;\\
g_{1,1/4}^{(1)}(\k_2)=&-9\k_2^2(\k_2^2-1)^2 \leq 0;\\
g_{1,1/4}(\k_2)=&-4\k_2^3(\k_2^2-1)^2\leq 0;
\end{align*}
and
\begin{align*}
g_{2,1/4}^{(2)}(\k_2)=&-\frac{1}{2}\k_2(-9-74\k_2^2+147\k_2^4)\leq 0;\\
g_{2,1/4}^{(1)}(\k_2)=&-\k_2^2(\k_2^2-1)(-3+19\k_2^2) \leq 0;\\
g_{2,1/4}(\k_2)=&-4\k_2^3(\k_2^2-1)^2\leq 0;
\end{align*}
\item If $\k_2\leq \frac{1}{\k_2}$, then $0<\k_2 \leq 1$ and
\begin{align*}
g_{1,1/4}^{(2)}(1/\k_2)=&-\frac{3}{2\k_2}(1-\k_2^2)(11-5\k_2^2-2\k_2^4)\leq 0;\\
g_{1,1/4}^{(1)}(1/\k_2)=&-\frac{3}{4\k_2^2}(1-\k_2^2)^2(3-2\k_2+\k_2^2)(3+2\k_2+\k_2^2)\leq 0;\\
g_{1,1/4}(1/\k_2)=&-\frac{2}{\k_2^3}(1-\k_2^2)^2(1+\k_2^4)\leq 0;
\end{align*}
and
\begin{align*}
g_{2,1/4}^{(2)}(1/\k_2)=&-\frac{1}{2\k_2^3}\left(95-42\k_2^2+27\k_2^4-16\k_2^6\right)\leq 0;\\
g_{2,1/4}^{(1)}(1/\k_2)=&-\frac{1}{4\k_2^4}(1-\k_2^2)(43-5\k_2^2+29\k_2^4-3\k_2^6)\leq 0;\\
g_{2,1/4}(1/\k_2)=&-\frac{2}{\k_2^5}(1-\k_2^2)^2(1+\k_2^4)\leq 0;
\end{align*}
\end{enumerate}
In both cases, we have $g_{i,1/4}(\k_1)\leq 0$.

Thus the coefficients $a_1,a_2$ in \eqref{5.7} are non-positive at the spatial maximum point of $G$.  Applying the maximum principle, we conclude that $\max_{M_t}G(x,t)$ is monotone non-increasing along the flow \eqref{5.1} with $\a \in [\frac{1}{4},1]$, which completes the proof.
\end{proof}

Applying similar argument as in Corollary \ref{s4:cor}, we have
\begin{cor}
Let $M_t, t\in [0,T)$ be a smooth solution of the flow \eqref{5.1} with $K>1$ in $\H^3$, where $\a\in [1/2,1]$. There exists a constant $C>0$ depending only on $M_0$ and $\alpha$ such that
\begin{align*}
0<\frac{1}{C}\leq \frac{\k_1}{\k_2} \leq C
\end{align*}
on $M_t$ for all $t\in [0,T)$.
\end{cor}

\section{Contracting flows in the sphere}\label{sec:S}
In this section, we study the flows for strictly convex surfaces in the sphere $\mathbb{S}^3$. We will prove the curvature pinching estimates along the flow.
\subsection{Flow by powers of mean curvature}
In this subsection, we study the flow of closed surfaces in the sphere by powers of mean curvature
\begin{align}\label{7.1}
\left\{\begin{aligned}\frac{\partial}{\partial t} X(x,t)=&-H^\a(x,t) \nu(x,t),\quad \alpha>0\\
X(x,0)=&X_0(x).
\end{aligned}\right.
\end{align}
We first show that the strict convexity of $M_t$ is preserved along the flow \eqref{7.1}.

\begin{pro}\label{pro-3.4}
Let $M_t$, $t\in [0,T)$ be a smooth solution to the flow \eqref{7.1} in $\mathbb{S}^3$. For any power $\a>0$, if $M_0$ is strictly convex, then $M_t$ is strictly convex for all $t\in [0,T)$. Moreover, the mean curvature satisfies
\begin{equation}\label{s3:H-S1}
  H(x,t)\geq \min_{M_0}H(\cdot,0)\left(1-\frac{\alpha+1}n(\min_{M_0}H(\cdot,0))^{\alpha+1}t\right)^{-\frac 1{\alpha+1}}
\end{equation}
for all $t\in [0,T)$. This implies an upper bound for the maximum existence time $T$:
\begin{equation*}
  T\leq \frac n{\alpha+1}(\min_{M_0}H(\cdot,0))^{-(\alpha+1)}.
\end{equation*}
\end{pro}
\begin{proof}
To show the strict convexity of $M_t$ for $t>0$, we prove that the Gauss curvature $K>0$ is preserved along the flow \eqref{7.1}. Let $G=K$ in \eqref{2.5}. Similar as in Proposition \ref{pro-3.1}, the zero order term $Q_0$ of the evolution equation of $K$ satisfies
\begin{align*}
Q_0=&(1-\a)KH^{\a+1}+2\a K (|A|^2-2) H^{\a-1}+(\a+1)H^{\a+1}\\
   =&(K+1)(H^2+(\k_1-\k_2)^2\a)H^{\a-1} \geq 0.
\end{align*}
The gradient term $Q_1$ is the same as in the hyperbolic case. Hence, by \eqref{s3:2} we have $Q_1\geq 0$ at the spatial minimum point of $K$. By the maximum principle, we have $\min_{M_t}K \geq \min_{M_0}K>0$, and hence $M_t$ is strictly convex for $t\in [0,T)$.

By the equation \eqref{s2:speed}, the mean curvature $H$ evolves along the flow \eqref{7.1} by
\begin{align*}
  \frac{\partial}{\partial t}H=&\Delta H^{\alpha}+H^{\alpha}(|A|^2+2).
\end{align*}
This implies that
\begin{align*}
  \frac{d}{d t}\min_{M_t}H\geq &H^{\alpha+2}/n.
\end{align*}
Then the estimate \eqref{s3:H-S1} follows from the maximum principle.
\end{proof}


As we mentioned in \S \ref{sec:1}, we consider the function
\begin{equation}\label{s3:G-S}
  G(x,t)=g(\kappa)=\frac{(\k_1-\k_2)^2(\k_1+\k_2)^{2\a}}{(\k_1\k_2)^2}
\end{equation}
to derive the curvature pinching estimate of the flow in the sphere $\mathbb{S}^3$. This is the same one used in \cite{Schulze2006} in the Euclidean case.
\begin{theo}\label{theo-3.4}
Let $M_t$, $t\in [0,T)$ be a smooth strictly convex solution of the flow \eqref{7.1} in $\mathbb{S}^3$, where $\a\in [1,5]$. Then $\max_{M_t}G(x,t)$ is monotone non-increasing in time.
\end{theo}
\begin{proof}
We apply \eqref{2.2} to calculate the evolution equation of $G$ along the flow \eqref{7.1}.  We have
\begin{equation}\label{7.4}
\frac{\partial}{\partial t}G=\dot{F}^{ij}\nabla_i \nabla_j G + Q_1+Q_0,
\end{equation}
at the spatial maximum point of $G$, where $Q_1,Q_0$ denote the gradient term and zero order term in the evolution of $G$. By the definition \eqref{s3:G-S}, the derivatives of $G$ are given by
\begin{align*}
\dot{g}^1=&~2\k_1^{-3}\k_2^{-2}(\k_1-\k_2)H^{2\a-1}(\k_2(\k_1+\k_2)+\k_1(\k_1-\k_2)\a),\\
\dot{g}^2=&~2\k_1^{-2}\k_2^{-3}(\k_2-\k_1)H^{2\a-1}(\k_1(\k_1+\k_2)+\k_2(\k_1-\k_2)\a).
\end{align*}
Then a direct calculation gives that
\begin{equation*}
\begin{split}
Q_0
=&-2 H^{\a-1}K^{-1} G(H^2+\a(\k_1-\k_2)^2) \leq 0.
\end{split}
\end{equation*}
Since the function $G$ is the same one used in \cite{Schulze2006} for the flow in $\mathbb{R}^3$ by powers of mean curvature, the gradient term $Q_1$ in \eqref{7.4} would be the same as in the Euclidean case. Therefore by the argument as in \cite[Lemma A.2]{Schulze2006}, $Q_1\leq 0$ at the spatial maximum point of $G$ provided that the power $\a \in [1, 5]$, and the conclusion of Theorem \ref{theo-3.4} follows.
\end{proof}

\begin{cor}\label{cor3.6}
Let $M_t, t\in [0,T)$ be a smooth strictly convex solution of $H^\a$-flow in $\mathbb{S}^3$, where $\a\in [1,5]$. There exists a constant $C>0$ depending only on the initial surface $M_0$ and $\alpha$ such that
\begin{align}\label{s3:pinc-2}
0<\frac{1}{C}\leq \frac{\k_1}{\k_2} \leq C,
\end{align}
on $M_t$ for all $t\in [0,T)$.
\end{cor}
\begin{proof}
\begin{align*}
\frac{\k_1}{\k_2}+ \frac{\k_2}{\k_1}-2=\frac{(\k_1-\k_2)^2}{\k_1\k_2}=\frac{(\k_1-\k_2)^2H^{2\a}}{(\k_1\k_2)^{2}} \cdot \frac{\k_1\k_2}{H^{2\a}}
\leq  &\frac{\max_{M_0}G(\cdot,0)}{4H^{2\a-2}},
\end{align*}
which is bounded from above by \eqref{s3:H-S1} and $\a\geq 1$. Hence, the estimate \eqref{s3:pinc-2} follows immediately.
\end{proof}

\subsection{Flow by powers of Gauss curvature}
In this subsection, we study the flow for surfaces in the sphere by powers of Gauss curvature, i.e.,
\begin{align}\label{6.1}
\left\{\begin{aligned}\frac{\partial}{\partial t} X(x,t)=&-K^\a(x,t) \nu(x,t),\qquad \alpha>0\\
X(x,0)=&X_0(x).
\end{aligned}\right.
\end{align}

\begin{pro}\label{pro-5.4}
Let $M_t$, $t\in [0,T)$ be a smooth solution to the flow \eqref{6.1} in $\mathbb{S}^3$. For any power $\a>0$, if $M_0$ is strictly convex, then $M_t$ is strictly convex for all $t\in [0,T)$. Moreover,we have the estimate
\begin{equation}\label{s5:K-S1}
K\geq \min_{M_0}K\left(1-(2\alpha+1)\min_{M_0}K^{\alpha+\frac 12}t\right)^{-\frac 2{2\alpha+1}}
\end{equation}
on $M_t$ for $t\in [0,T)$.
\end{pro}
\begin{proof}
This follows from a similar argument as in Proposition \ref{pro-5.1}. By the equation \eqref{s2:speed}, the speed function $F=K^{\alpha}$ of the flow \eqref{6.1} evolves by
\begin{align}\label{s5:K-S2}
\frac{\partial}{\partial t}K^{\alpha}=&~\alpha K^{\alpha-1}\dot{K}^{ij}\nabla_i\nabla_jK^{\alpha}+\alpha K^{2\alpha-1}H(K+1).
\end{align}
Since $\alpha>0$, $K>0$ is preserved by applying the maximum principle to \eqref{s5:K-S2}.  The equation \eqref{s5:K-S2} also implies that
\begin{equation*}
  \frac d{dt}\min_{M_t}K\geq 2K^{\alpha+\frac 32}.
\end{equation*}
Then the estimate \eqref{s5:K-S1} follows by integrating the above inequality.
\end{proof}

We consider the following function
\begin{equation}\label{s5:G-2}
  G(x,t)=g(\kappa)=\frac{(\k_1-\k_2)^2}{(\k_1\k_2)^{2-2\a}}.
\end{equation}
\begin{theo}\label{theo-5.4}
Let $M_t$, $t\in [0,T)$ be a smooth strictly convex solution of the flow \eqref{6.1} in $\mathbb{S}^3$, where $\a\in [1/4,1]$. Then $\max_{M_t}G(x,t)$ is monotone non-increasing in time.
\end{theo}
\begin{proof}
Since $F=f(\k)=(\k_1\k_2)^{\a}$ and $G$ is defined as in \eqref{s5:G-2}, the derivatives of $F$ and $G$ are given by
\begin{equation}\label{6.4}
\begin{split}
\dot{f}^1=&\a \k_1^{-1}(\k_1\k_2)^{\a}, \quad  \dot{f}^2=\a \k_2^{-1}(\k_1\k_2)^{\a},\\
\dot{g}^1=&\frac{2}{\k_1^3\k_2^2}(\k_1-\k_2)(\k_1\k_2)^{2\a}(\k_2+(\k_1-\k_2)\a),\\
\dot{g}^2=&\frac{2}{\k_1^2\k_2^3}(\k_2-\k_1)(\k_1\k_2)^{2\a}(\k_1+(\k_2-\k_1)\a).
\end{split}
\end{equation}
We apply \eqref{2.2} to calculate the evolution equation of $G$ along the flow \eqref{6.1}, i.e., at the spatial maximum point of $G$ we have:
\begin{equation}\label{6.3}
\frac{\partial}{\partial t}G=\dot{F}^{ij}\nabla_i \nabla_j G + Q_1+Q_0,
\end{equation}
where $Q_1,Q_0$ denotes the gradient term and zero order term.  By \eqref{2.2} and \eqref{6.4}, the zero-order term of \eqref{6.3} for $G$ satisfies
\begin{equation}\label{6.6}
\begin{split}
Q_0
=&-2HK^{-1}fg \leq 0.
\end{split}
\end{equation}
Next we apply \eqref{2.13} to calculate the gradient term $Q_1$ at the maximum point $(p,t)$.
Note that the derivation of \eqref{2.13} didn't use any ambient curvatures, and so can be applied here. Taking further derivatives to the equations \eqref{6.4}, we have
\begin{equation}\label{6.7}
\begin{split}
\ddot{f}^{11}=&\a(\a-1)\k_1^{-2}K^\a, \quad \ddot{f}^{22}=\a(\a-1)\k_2^{-2}K^\a, \quad \ddot{f}^{12}=\a^2 K^{\a-1},\\
\ddot{g}^{11}=&2K^{2\a-2}\k_1^{-2}\left[\k_2(-2\k_1+3\k_2)-(\k_1-5\k_2)(\k_1-\k_2)\a+2(\k_1-\k_2)^2\a^2\right],\\
\ddot{g}^{22}=&2K^{2\a-2}\k_2^{-2}\left[\k_1(-2\k_2+3\k_1)-(\k_2-5\k_1)(\k_2-\k_1)\a+2(\k_1-\k_2)^2\a^2\right],\\
\ddot{g}^{12}=&2K^{2\a-3}\left[ -\k_1\k_2-2(\k_1-\k_2)^2 \a+2(\k_1-\k_2)^2\a^2\right].
\end{split}
\end{equation}
We only calculate the coefficient (denoted by $\mathcal{Z}$) in front of $T_1^2$ in the equation \eqref{2.13}.
Using \eqref{6.4} and \eqref{6.7}, we obtain the expression for $\mathcal{Z}$ as follows:
\begin{align*}
\mathcal{Z}=&2\a K^{3\a-1}\k_1^{-2}\biggl(\k_2^3(\a-1)+\k_1^3(\a-1)(4\a-1) \\
& \quad +\k_1\k_2^2(\a-1)(4\a+1)+\k_1^2\k_2(-3+7\a-8\a^2) \biggr).
\end{align*}
For $\a \in [\frac{1}{4},1]$, it is easy to check that $\mathcal{Z} \leq 0$. This means that the gradient term $Q_1$ in the evolution equation \eqref{6.3} is non-positive at the spatial maximum point of $G$, and then the conclusion of Theorem \ref{theo-5.4} follows.
\end{proof}
\begin{rem}
The gradient term in \eqref{6.3} is same as the Euclidean case in \cite{Andrews-Chen2012}. The computation in \cite{Andrews-Chen2012} is carried out using the Gauss map parametrization of the flow: The flow \eqref{6.1} in Euclidean space $\mathbb{R}^3$ is equivalent to a scalar parabolic equation on the sphere $\mathbb{S}^2$ for the support function of the evolving surfaces. Here we proved our estimate using the calculation on the evolving surfaces directly.
\end{rem}

\begin{cor}
Let $M_t, t\in [0,T)$ be a smooth strictly convex solution of the flow \eqref{6.1} in $\mathbb{S}^3$, where $\a\in [1/2,1]$. There exists a constant $C>0$ depending only on $M_0$ and $\alpha$ such that
\begin{align}\label{s5:pinc-2}
0<\frac{1}{C}\leq \frac{\k_1}{\k_2} \leq C
\end{align}
on $M_t$ for all $t\in [0,T)$.
\end{cor}
\begin{proof}
As before, we have
\begin{align*}
\frac{\k_1}{\k_2}+ \frac{\k_2}{\k_1}-2=\frac{(\k_1-\k_2)^2}{(\k_1\k_2)^{2-2\a}} \cdot \frac{1}{(\k_1\k_2)^{2\a-1}}\leq ~\max_{M_0}G(\cdot,0)\cdot \frac{1}{(\k_1\k_2)^{2\a-1}},
\end{align*}
which is bounded from above by the estimate \eqref{s5:K-S1} and $\alpha\geq 1/2$. Then the estimate \eqref{s5:pinc-2} follows.
\end{proof}

\section{Convergence}\label{sec:6}
In this section, we discuss the convergence of the solution to a point and of the rescaled solution to a sphere. We only give the details for the flow in $\mathbb{H}^3$ by powers of mean curvature, since the proof is similar for the remaining flows.
\subsection{Contraction to a point}
\begin{pro}\label{s6:lem1}
The evolving surfaces $M_t$ of the flow \eqref{3.1} remain smooth until they contract to a point as $t\to T$.
\end{pro}
\proof
Let $\rho_+(t)$ and $\rho_-(t)$ be the outer radius and inner radius of the domain $\Omega_t$ enclosed by $M_t$, defined by
\begin{align*}
  \rho_+(t) =& \inf\{\rho: ~\Omega_t\subset B_{\rho}(p)~\mathrm{for~ some} ~p\in \mathbb{H}^3\} \\
\rho_-(t) =& \sup\{\rho: ~B_{\rho}(p)\subset \Omega_t~\mathrm{for~ some} ~p\in \mathbb{H}^3\}.
\end{align*}
By the pinching estimate \eqref{s3:pinc-1}, we can apply a similar argument in \cite[\S 6]{Ger15} (see also \cite[\S 3]{AHL19}) to show that
\begin{equation}\label{s6:pinc-1}
  \rho_+(t)\leq C\rho_-(t),\qquad \mathrm{for} ~t\in [t_0,T),
\end{equation}
where $t_0$ is sufficiently close to $T$.

The technique of Tso \cite{Tso1985} can be used to show that the mean curvature remains bounded as long as the flow encloses a non-vanishing volume: Assume that there exits a geodesic ball $B_{\rho}(x_0)\subset \Omega_t$ for $t\in [0,t_1]$, where $t_1\in [t_0,T)$. Since $M_t$ is strictly convex, we can write $M_t=\mathrm{graph}~ u(\cdot,t)$ as graphs in polar coordinates centered at $x_0$. Then $u\geq \rho$ for all $t\in [0,t_1]$. By the comparison principle, the later surface is contained in the earlier one, then we have an upper bound on $u\leq 2\rho_+(0)$ depending only on $M_0$. Denote by $\partial_r$ the gradient vector at $x\in M_t$ along the geodesic from ${x_0}$ to $x$. The support function of $M_t$  with respect to $x_0$ is defined by $\chi(x,t)=\sinh u(x,t)\langle \partial_r,\nu\rangle$. Due to the strict convexity of $M_t$ and $\rho\leq u\leq 2\rho_+(0)$,  we have
$\chi\geq  \sinh \rho$ for all $t\in [0,t_1]$, and the function
\begin{equation*}
  \varphi~=~\frac{H^{\alpha}}{\chi-\frac 12 \sinh \rho}
\end{equation*}
is well defined on $M_t$ for all $t\in [0,t_1]$. Recall that by \eqref{s2:speed} the mean curvature satisfies
\begin{align}\label{s6:H-evl-1}
\frac{\partial}{\partial t}H^{\alpha}=~\alpha H^{\alpha-1} \D H^{\alpha}+\alpha H^{2\alpha-1}(|A|^2-2),
\end{align}
and the support function satisfies (see \cite[\S 4]{AW18})
\begin{equation*}
  \frac{\partial}{\partial t}\chi=\alpha H^{\alpha-1}\Delta \chi+\alpha H^{\alpha-1}|A|^2\chi-(1+\alpha)H^{\alpha}\cosh u(x).
\end{equation*}
The function $\varphi$ satisfies the evolution equation
\begin{align}\label{s6:var}
  \frac{\partial}{\partial t}\varphi= & \alpha H^{\alpha-1}\left(\Delta \varphi+\frac 2{\chi-\frac 12 \sinh \rho}g^{ij}\nabla_i\chi\nabla_j\varphi\right) \nonumber\\
   & \quad +\left((1+\alpha)\cosh u(x)-\frac 12 \sinh \rho \alpha \frac{|A|^2}H\right)\varphi^2-2\alpha H^{\alpha-1}\varphi \nonumber\\
   \leq &\alpha H^{\alpha-1}\left(\Delta \varphi+\frac 2{\chi-\frac 12 \sinh \rho}g^{ij}\nabla_i\chi\nabla_j\varphi\right) \nonumber\\
   & \quad +\left((1+\alpha)\cosh u(x)-\frac {\alpha }2(\frac 12 \sinh \rho)^{\frac{\alpha+1}{\alpha}} \varphi^{1/{\alpha}}\right)\varphi^2,
\end{align}
where we used $|A|^2\geq H^2/2$ and $\chi\geq \sinh \rho$. Applying maximum principle to \eqref{s6:var} gives the upper bound on $\varphi$. This together with the upper bound $\chi\leq \sinh(2\rho_+(0))$ implies that $H$ is bounded from above by a constant depending on $\rho,\alpha,M_0$.

On the other hand, by Proposition \ref{pro-3.1} we have $H\geq 2\sqrt{K}>2$. This together with the upper bound on mean curvature and the pinching estimate \eqref{s3:pinc-1} implies that all the principal curvatures of $M_t$ are bounded above and below by positive constants. In particular, the coefficients $\dot{F}^{ij}=\a H^{\a-1} g^{ij}$ in the second order part of the problem have eigenvalues bounded above and below by positive constants, and then the flow remains to be uniformly parabolic. By applying the H\"{o}lder estimate of the second derivatives of uniformly parabolic equation of Andrews \cite{Andrews2004}, and standard Schauder theory, we can derive the higher regularity estimates of the solution to the flow. It follows that the solution can be extended past time $t_1$. This means that the smooth solution of the flow \eqref{3.1} exists as long as the evolving domain encloses a non-vanishing volume. Therefore the inner radius $\rho_-(t)\to 0$ as $t\to T$. The estimate \eqref{s6:pinc-1} then says that the outer radius converges to zero as $t\to T$ as well. In other words, the flow remains smooth until it contracts to a point as $t\to T$.
\endproof

\subsection{Convergence of the rescaled solution}
To study the asymptotical behavior of the flow, we consider the rescaling of the solution by adapting a similar procedure in \cite{Ger15} for the contracting flow in the sphere. If the initial surface is a geodesic sphere in $\mathbb{H}^3$, the evolving surfaces $M_t$ are all geodesic spheres with the same center and radius $\Theta=\Theta(t,T)$ satisfying \begin{equation*}
  \frac d{dt}\Theta=-2^{\alpha}\coth^{\alpha}\Theta.
\end{equation*}
The spherical solution shrinks to a point in finite time, and we choose the initial sphere such that the maximum existence time of the shrinking sphere is equal to the maximum existence time $T$ as in Lemma \ref{s6:lem1}.

We define the rescaled mean curvature by $\tilde{H}=\Theta(t,T)H$. Since we only care about the asymptotical behavior of the flow, we may focus on the flow in the time interval $[t_0,T)$, where $t_0$ is the time such that the pinching estimate \eqref{s6:pinc-1} holds in $[t_0,T)$.
\begin{lem}\label{s6:lem-H-ub}
There exists a uniform constant $C$ such that
\begin{equation}\label{s6:H-ub}
  \tilde{H}\leq C,\qquad \forall~t\in [t_0,T).
\end{equation}
\end{lem}
\proof
For any $t_1\in (t_0,T)$, let $B_{\rho_-(t_1)}(x_1)$ be an inball of $M_{t_1}$. Write $M_t=\mathrm{graph}~u(x,t)$ as geodesic radial graphs with respect to the point $x_1$ for all $t\in [t_0,t_1]$. Then $\rho_-(t_1)\leq u(t_1)\leq u(t)\leq R=2\rho_+(0)$. Since $M_t$ is strictly convex, we have $\chi\geq \sinh\rho_-(t_1)$ for all $t\in [t_0,t_1]$. Applying maximum principle to \eqref{s6:var}, we have
\begin{align*}
  \varphi(t) \leq & ~\max\biggl\{\varphi(t_0)\left(1+\frac{\alpha+1}4 (\frac 12 \sinh \rho_-(t_1))^{\frac{\alpha+1}{\alpha}}\varphi(t_0)^{\frac{\alpha+1}{\alpha}}(t-t_0)\right)^{-\frac{\alpha}{1+\alpha}},\\
   & \qquad \quad \left(\frac{4(1+\alpha)}{\alpha}\cosh R\right)^{\alpha}(\frac 12 \sinh \rho_-(t_1))^{-\alpha-1}\biggr\}.
\end{align*}
Choosing $t_1$ sufficiently close to $T$, we can make $\rho_-(t_1)$ small enough such that
\begin{align*}
  \varphi(t_1) \leq & ~\left(\frac{4(1+\alpha)}{\alpha}\cosh R\right)^{\alpha}(\frac 12 \sinh \rho_-(t_1))^{-\alpha-1}.
\end{align*}
Then
\begin{align*}
H^{\alpha}(t_1)(\sinh \rho_-(t_1))^{\alpha} =&\varphi(t_1)(\chi-\frac 12 \sinh \rho_-(t_1))(\sinh \rho_-(t_1))^{\alpha}\\
\leq & ~\left(\frac{4(1+\alpha)}{\alpha}\cosh R\right)^{\alpha}(\frac 12 \sinh \rho_-(t_1))^{-1}(\chi-\frac 12 \sinh \rho_-(t_1))\\
\leq &~\left(\frac{4(1+\alpha)}{\alpha}\cosh R\right)^{\alpha}(\frac 12 \sinh \rho_-(t_1))^{-1}(\sinh2\rho_+(t_1)-\frac 12 \sinh \rho_-(t_1))\\
\leq &~ C,
\end{align*}
where in the last inequality we used the estimate \eqref{s6:pinc-1}. The estimate \eqref{s6:H-ub} follows because $\Theta(t_1,T)$ is comparable with $\sinh\rho_-(t_1)$ for $t_1$ sufficiently close to $T$. Indeed,  by the comparison principle the spherical solution of radius $\Theta(t,T)$ must intersect $M_t$ for $t\in [0,T)$. This implies that $\inf_{M_t}u(\cdot,t)\leq \Theta(t,T)\leq \sup_{M_t}u(\cdot,t)$. Combining this with the pinching estimate \eqref{s6:pinc-1}, we have $\rho_-(t_1)\leq \Theta(t_1,T)\leq 2\rho_+(t_1)\leq 2C\rho_-(t_1)$. For $t_1$ sufficiently close to $T$, this is equivalent to that $\Theta(t_1,T)$ is comparable with $\sinh\rho_-(t_1)$.
\endproof

Let us define a new time parameter $\tau=-\log \Theta$. As $\Theta(t,T) \to 0$ as $t\to T$, we have $\tau$ ranges from $0$ to $\infty$. Then
\begin{equation*}
  \frac {d\tau}{dt}=-\frac 1{\Theta}\frac d{dt}\Theta=2^{\alpha}\Theta^{-1}\coth^{\alpha}\Theta.
\end{equation*}
From \eqref{s6:H-evl-1} we obtain
\begin{align}\label{s6:H-td1}
  \frac{\partial}{\partial \tau} \tilde{H}=& \frac{\partial}{\partial t}(\Theta H)\frac{dt}{d\tau}= 2^{-\alpha}\tanh^{\alpha}(\Theta)\Theta^2\frac{\partial}{\partial t}H-\tilde{H}\nonumber\\
  =& 2^{-\alpha}\tanh^{\alpha}(\Theta)\Theta^2\left(\Delta H^{\alpha}+ H^{\alpha}(|A|^2-2)\right)-\tilde{H}\nonumber\\
  =&2^{-\alpha}\Theta^{-\alpha}\tanh^{\alpha}(\Theta)\Theta^2\Delta \tilde{H}^{\alpha}+2^{-\alpha}\Theta^{-\alpha}\tanh^{\alpha}(\Theta)\tilde{H}^{\alpha}(|\tilde{A}|^2-2\Theta^2)-\tilde{H},
\end{align}
where $\tilde{A}=\Theta A$ denotes the rescaled second fundamental form. The upper bound \eqref{s6:H-ub} together with the pinching estimate \eqref{s3:pinc-1} implies that the rescaled principal curvatures $\tilde{\kappa}_i=\Theta \kappa_i$ are uniformly bounded from above.

Let $t_1\in [t_0,T)$ be arbitrary and let $t_2>t_1$ such that
\begin{equation*}
  \Theta(t_1,T)=2\Theta(t_2,T).
\end{equation*}
Then $\tau_i=-\log\Theta(t_i,T)$ satisfies $\tau_2=\tau_1+\log 2$. Introduce polar coordinates with respect to the center of an inball of $\Omega_{t_2}$ and write $M_t$ as graphs of $u(x,t)$ for $t\in [t_1,t_2]$. Then the pinching estimate \eqref{s6:pinc-1} implies
\begin{equation}\label{s6:uTheta}
  C^{-1}\Theta(t,T)\leq u(x,t)\leq C\Theta(t,T),\qquad \forall~t\in [t_1,t_2]
\end{equation}
and $ u_{\max}(t)\leq C^2u_{\min}(t)$,  for all $t\in [t_1,t_2]$. Since $M_t$ is strictly convex, this implies that
\begin{equation}\label{s6:v}
  v^2=1+\frac{|Du|^2_{g_{\mathbb{S}^2}}}{\sinh^2u}
\end{equation}
is uniformly bounded in $[t_1,t_2]\times \mathbb{S}^2$ (see \cite[Theorem 2.7.10]{Ger06}).
\begin{lem}\label{s7:lem3}
The rescaled mean curvature $\tilde{H}$ satisfies the porous medium type equation
\begin{align}\label{s6:H-td2}
  \frac{\partial}{\partial \tau} \tilde{H}=& \bar{\nabla}_i(a^{ij}\bar{\nabla}_j\tilde{H}^{\alpha})+b^i\partial_i\tilde{H}^{\alpha}+c \tilde{H}
\end{align}
in the cylinder $Q(\tau_1,\tau_2)=[\tau_1,\tau_2]\times \mathbb{S}^2$ with uniformly bounded coefficients $b^i$ and $c$, and
\begin{equation}\label{s6:a}
  C^{-1}\leq (a^{ij})\leq C
\end{equation}
 independent of $\tau_i$. Here $\tau_i=-\log\Theta(t_i,T)$ satisfies $\tau_2=\tau_1+\log 2$, $\bar{\nabla}$ denotes the covariant derivative on $\mathbb{S}^2$ with respect to the standard metric $g_{\mathbb{S}^2}=(\sigma_{ij})$.
\end{lem}
\proof
By \eqref{s6:H-td1}, the evolution of $\tilde{H}$ satisfies
\begin{align}
  \frac{\partial}{\partial \tau} \tilde{H}=& 2^{-\alpha}\Theta^{-\alpha}\tanh^{\alpha}(\Theta)\Theta^2\nabla_i(g^{ij}\nabla_j \tilde{H}^{\alpha})\nonumber\\
  &\quad + 2^{-\alpha}\Theta^{-\alpha}\tanh^{\alpha}(\Theta)\tilde{H}^{\alpha}(|\tilde{A}|^2-2\Theta^2)-\tilde{H}\nonumber\\
  =& 2^{-\alpha}\Theta^{-\alpha}\tanh^{\alpha}(\Theta)\left(\bar{\nabla}_i(\Theta^2g^{ij}\bar{\nabla}_j \tilde{H}^{\alpha})+\Theta^2g^{ij}(\Gamma_{ij}^k-\bar{\Gamma}_{ij}^k)\partial_k\tilde{H}^{\alpha}\right)\nonumber\\
  &\quad + 2^{-\alpha}\Theta^{-\alpha}\tanh^{\alpha}(\Theta)\tilde{H}^{\alpha}(|\tilde{A}|^2-2\Theta^2)-\tilde{H},
\end{align}
where $g^{ij}$ is the inverse of the metric $g_{ij}=u_iu_j+\sinh^2u\sigma_{ij}=\sinh^2u(\varphi_i\varphi_j+\sigma_{ij})$, $\Gamma_{ij}^k$ and $\bar{\Gamma}_{ij}^k$ are Christoffel symbols of the metric $g_{ij}$ and $\sigma_{ij}$ respectively. Here $\varphi$ is defined such that
\begin{equation*}
  \varphi_i=\frac{u_i}{\sinh u}.
\end{equation*}
By the estimate \eqref{s6:v} on $v$, $\varphi_i$ is uniformly bounded. Hence $g^{ij}\Theta^2\approx \sinh^{-2}u\Theta^2$ is uniformly bounded from above and from below by positive constants in view of \eqref{s6:uTheta}. Since $\Theta$ is small in the interval $[t_0,T)$, we also have uniform bound on $\Theta^{-\alpha}\tanh^{\alpha}\Theta$. This gives the estimate \eqref{s6:a} on the coefficients $a^{ij}$. To estimate the bound on $b^i$, we notice that
\begin{equation*}
\Gamma_{ij}^k-\bar{\Gamma}_{ij}^k=\frac 12 g^{kq}(\bar{\nabla}_ig_{jq}+\bar{\nabla}_jg_{iq}-\bar{\nabla}_qg_{ij}),
\end{equation*}
which depends on the first and second derivatives of $\varphi$. Recall that the Weingarten matrix of the graph $M_t=\mathrm{graph} ~u(x,t)$ is given by (see \cite{Ger06})
\begin{equation}\label{s6:h}
  h_i^j=-\frac 1{v\sinh u}(\sigma^{ik}-\frac{\varphi^i\varphi^k}{v^2})\varphi_{jk}+\frac{\coth u}v\delta_i^j.
\end{equation}
Since the rescaled Weingarten matrix $\tilde{h}_i^j=\Theta h_i^j$ and $\varphi_i$ are uniformly bounded, The equation \eqref{s6:h} gives the upper bound on $\varphi_{ij}$, the second derivatives of $\varphi$. Then $\Gamma_{ij}^k-\bar{\Gamma}_{ij}^k$ is uniformly bounded. Finally, the bound on the zero order term $c$ follows from upper bound on the rescaled principal curvatures. This completes the proof of the lemma.
\endproof

We can apply \cite[Theorem 1.2]{DiBenedetto-Friedman1985} to \eqref{s6:H-td2} to deduce the H\"older continuity estimate for $\tilde{H}$ on the region $Q(\tau_1,\tau_2)$, with the constant depending on $\int_{Q(\tau_1,\tau_2)} |\bar{\nabla} \tilde{H}^\a|^2 d\mu_{\mathbb{S}^2}d\tau$. To bound this term,  by \eqref{s6:H-evl-1} and integration by parts, we have
\begin{align*}
\frac{d}{dt}\int_{M_{t}} H^{\a+1}d\mu_t=&~(\a+1) \int_{M_{t}} H^\a \left(\Delta H^{\a}+H^\a(|A|^2-2)\right)d\mu_{t}-\int_{M_{t}} H^{2(\a+1)}d\mu_{t} \\
\leq &~\alpha \int_{M_{t}} H^{2(\a+1)}d\mu_{t}-(\a+1) \int_{M_{t}} |\nabla H^{\a}|^2d\mu_{t},
\end{align*}
where we used $|A|^2\leq H^2$ since each $M_t$ is strictly convex. Equivalently,
\begin{align*}
\frac{d}{d\tau}\int_{M_{t}} H^{\a+1}d\mu_t=&~\frac{d}{dt}\int_{M_{t}} H^{\a+1}d\mu_t\cdot \frac {dt}{d\tau}\\
\leq &~2^{-\alpha}\Theta\tanh^{\alpha}\Theta\left(\alpha\int_{M_{t}} H^{2(\a+1)}d\mu_{t}-(\a+1) \int_{M_{t}} |\nabla H^{\a}|^2d\mu_{t}\right),
\end{align*}
where $\tau=-\log\Theta(t,T)$. This implies that
\begin{align}\label{s6:dH}
  \int_{\tau_1}^{\tau_2}\int_{M_t}|\nabla H^{\alpha}|^2d\mu d\tau\leq & -\frac{2^{\alpha}}{\alpha+1}\int_{\tau_1}^{\tau_2}\frac 1{\Theta\tanh^\alpha\Theta}\frac{d}{d\tau}\int_{M_{t}} H^{\a+1} \nonumber\\
   &+ \frac{\alpha}{\alpha+1}\int_{\tau_1}^{\tau_2}\int_{M_t}H^{2(\alpha+1)}d\mu d\tau\nonumber\\
   =&-\frac{2^{\alpha}}{\alpha+1}\left(\frac 1{\Theta\tanh^\alpha\Theta}\int_{M_{t}} H^{\a+1}\right) \biggr|_{\tau_1}^{\tau_2}\nonumber\\
  & +\frac{2^{\alpha}}{\alpha+1}\int_{\tau_1}^{\tau_2}\left(\frac 1{\Theta\tanh^\alpha\Theta}
   +\frac{\alpha}{\tanh^{(\alpha+1)}\Theta\cosh^2\Theta}\right)\int_{M_t}H^{2(\alpha+1)}d\mu d\tau\nonumber\\
   &+ \frac{\alpha}{\alpha+1}\int_{\tau_1}^{\tau_2}\int_{M_t}H^{2(\alpha+1)}d\mu d\tau.
\end{align}
Since $g_{ij}\approx \Theta^2\sigma_{ij}$,
\begin{equation*}
  \int_{\tau_1}^{\tau_2}\int_{\mathbb{S}^2}|\bar{\nabla} \tilde{H}^{\alpha}|_{g_{\mathbb{S}^2}}^2d\mu_{\mathbb{S}^2}d\tau ~\lesssim~ \Theta^{2\alpha}\int_{\tau_1}^{\tau_2}\int_{M_t}|\nabla H^{\alpha}|^2d\mu d\tau.
\end{equation*}
Multiplying the two sides of \eqref{s6:dH} by $\Theta^{2\alpha}$ and using the estimate \eqref{s6:H-ub}, we obtain the required bound on $\int_{Q(\tau_1,\tau_2)} |\bar{\nabla} \tilde{H}^\a|^2 d\mu_{\mathbb{S}^2}d\tau$. Thus, applying Theorem 1.2 in \cite{DiBenedetto-Friedman1985}, we obtain
\begin{lem}
For any $(x,\tau)\in \mathbb{S}^2\times [\tau_0+\log 2,\infty)$, there exist a universal constant $\delta>0$ and some $\gamma<1$ such that the $\gamma$-H\"{o}lder norm of $\tilde{H}$ on the space-time neighborhood $B_{\delta}(x)\times [\tau-\delta,\tau+\delta]$ is uniformly bounded, where $B_{\delta}(x)$ denotes a geodesic ball of radius $\delta$ centered at $x$ in $\mathbb{S}^2$.
\end{lem}

Let $p\in \mathbb{H}^3$ be the point the flow surfaces are shrinking to. Introduce geodesic polar coordinates around $p$ and write $M_t$ as graphs of $u(x,t)$, $(x,t)\in \mathbb{S}^2\times [t_0,T)$. We consider the rescaled function $\tilde{u}(x,\tau)=u(x,t)\Theta^{-1}(t,T)$ on $(x,\tau)\in \mathbb{S}^2\times [\tau_0,\infty)$. Note that the rescaled principal curvatures $\tilde{\kappa}_i=\Theta\kappa_i$ are not the principal curvatures of the graph of $\tilde{u}$, though they are related. By the proof of Lemma \ref{s7:lem3}, the $C^2$-norm of $\tilde{u}$ is uniformly bounded. Thus for any sequence of time $\tau_j$, there exists a subsequence (still denoted by $\tau_j$) such that $\tilde{u}(\cdot,\tau_j)$ converges in $C^{1,\gamma}$ to a limit function $\tilde{u}_{\infty}$ for any $\gamma<1$.  At each time $\tau_j$, let $p_j\in \mathbb{S}^2$ be a point such that $\tilde{u}_{\max}(\tau_j)=\tilde{u}(p_j,\tau_j)$. Then $\tilde{\kappa}_i(p_j,\tau_j)=\kappa_i(p_j,t_j)\Theta(t_j,T)\geq \coth u(p_j,t_j)\Theta(t_j,T)\geq C>0$, where we recall that $\tau_j$ and $t_j$ are related by $\tau_j=-\log\Theta(t_j,T)$. This implies that $\tilde{H}(p_j,\tau_j)\geq 2C>0$. The H\"{o}lder continuity of $\tilde{H}$ implies that $\tilde{H}$ can not decrease too fast in the sense that $\tilde{H}\geq C$ in $B_{\delta}(p_j)\times [\tau_j-\delta,\tau_j+\delta]$. The rescaled function $\tilde{u}(x,\tau)$ now satisfies the uniformly parabolic equation
\begin{align}\label{s7:u-td}
  \frac{\partial}{\partial\tau}\tilde{u} =& -2^{-\alpha}\tanh^{\alpha}\Theta vH^{\alpha}+\tilde{u}\nonumber\\
  = & -2^{-\alpha}\Theta^{-\alpha}\tanh^{\alpha}\Theta v\tilde{H}^{\alpha}+\tilde{u}
\end{align}
in $B_{\delta}(p_j)\times [\tau_j-\delta,\tau_j+\delta]$, where $v$ is the function defined in \eqref{s6:v}. By the H\"{o}lder estimate \cite{Andrews2004} and Schauder estimate, we obtain uniform $C^{\infty}$ estimate for the rescaled function $\tilde{u}$ in $B_{\delta/2}(p_j)\times [\tau_j-\delta/2,\tau_j+\delta/2]$. Since the sphere $\mathbb{S}^2$ is compact, there exists a point $p_{\infty}\in \mathbb{S}^2$ such that after passing to a subsequence we have $p_j\to p_{\infty}$. The above estimate implies that $\tilde{u}(x,\tau_j)$ converges to $\tilde{u}_{\infty}$ in $C^{\infty}$ for all $x\in B_{\delta/2}(p_{\infty})$.

By Theorem \ref{theo-3.3},
\begin{align*}
\frac{\k_1}{\k_2}+\frac{\k_2}{\k_1}-2= &\frac{(\k_1+\k_2)^{2\a}(\k_1-\k_2)^2}{(\k_1\k_2-1)^2}\frac{(\k_1\k_2-1)^2}{\k_1\k_2(\k_1+\k_2)^{2\alpha}}\\
\leq&~C H^{2(1-\alpha)}=C\tilde{H}^{2(1-\alpha)}\Theta^{2(\alpha-1)}
\end{align*}
which converges to zero as $\tau_j\to\infty$ in $B_{\delta/2}(p_{\infty})$, because $\alpha>1$ and $\tilde{H}(x,\tau_j)$ is bounded in $B_{\delta}(p_j)$. In other words,
\begin{align}\label{s6:pinc-2}
1~\leq~ \frac{\k_1}{\k_2}~\leq ~1+C\Theta^{(\alpha-1)}=1+Ce^{-(\alpha-1)\tau_j}~ \to 1\qquad \mathrm{as} ~\tau_j\to\infty
\end{align}
in $B_{\delta/2}(p_{\infty})$. By the inequality $|\nabla H|^2\leq 4|\nabla A|^2/3$ (see \cite[\S 2]{Huisken1984}) and interpolation inequality for the unscaled quantity on $M_t$, we have
\begin{equation*}
  |\nabla A|^2\leq 3|\nabla \mathring{A}|^2\leq C|\mathring{A}||\nabla^2\mathring{A}|.
\end{equation*}
We deduce that
\begin{equation*}
  |\nabla \tilde{A}|^2=\Theta^2g^{ij}(\Theta h_{k;i}^l)(\Theta h_{l;j}^k) =\Theta^4|\nabla A|^2\leq  C|\mathring{\tilde{A}}||\nabla^2\mathring{\tilde{A}}|\leq C e^{-(\alpha-1)\tau_j}
\end{equation*}
converges to zero in $B_{\delta/2}(p_{\infty})$ as $\tau_j\to\infty$, where we used the fact that $|\nabla^2\mathring{\tilde{A}}|$ is bounded due to the regularity estimate of $\tilde{u}=u\Theta^{-1}$. This implies that
\begin{align*}
  (\tilde{H}_{\max}-\tilde{H}_{\min})|_{B_{\delta/2}(p_{\infty})} \leq & ~\Theta |\nabla A|\mathrm{diam}(M_t\cap ~\mathrm{graph} ~u(\cdot,t_j)|_{B_{\delta/2}(p_{\infty})}) \\
  = &~|\nabla\tilde{A}|\Theta^{-1} \mathrm{diam}(M_t\cap ~\mathrm{graph} ~u(\cdot,t_j)|_{B_{\delta/2}(p_{\infty})}) \\
  \leq & ~Ce^{-\frac 12(\alpha-1)\tau_j},
\end{align*}
where $\tau_j=-\log\Theta(t_j,T)$. Therefore, $\tilde{H}(x,\tau_j)$ becomes arbitrary close to the value $\tilde{H}(p_j,\tau_j)$ in $B_{\delta/2}(p_{\infty})$. Using the H\"{o}lder continuity of $\tilde{H}$ and repeating the above argument, we can extend the region where $\tilde{u}(x,\tau_j)$ converges in $C^{\infty}$ to $\tilde{u}_{\infty}$ to a larger one, say $B_{\delta}(p_{\infty})$. After a finite number of iterations, we deduce that $\tilde{u}(\cdot,\tau_j)$ converges in $C^{\infty}$ to $\tilde{u}_{\infty}$ on $\mathbb{S}^2$.  The above argument can be applied to any sequence $\tau_j$, we conclude that the whole family $\tilde{u}(\cdot,\tau)$ satisfy uniform $C^{\infty}$ estimates on $\mathbb{S}^2$, and converge to the same limit function $\tilde{u}_{\infty}$ smoothly as $\tau\to\infty$. Here the convergence to the same limit function follows from the evolution equation \eqref{s7:u-td} and Cauchy criterion. The interpolation inequality then implies that $|\nabla \tilde{A}|^2(\cdot,\tau)\leq Ce^{-(\alpha-1)\tau}$ converges to zero as $\tau\to\infty$ on $\mathbb{S}^2$. By a similar argument as in \cite[\S 8]{Ger15}, we can prove that $\tilde{u}\to 1$ as $\tau\to\infty$. The smooth convergence of $\tilde{u}$ to $1$ again follows from the interpolation inequality. Thus, we complete the proof of Theorem \ref{main-theo-I} for the flow by powers of mean curvature.

\begin{proof}[Proof of Theorem \ref{main-theo-I}]
We have given a detailed proof for the case (i). For case (ii), the flow by powers of scalar curvature. A similar argument as in case (i) implies that the scalar curvature $R=2(K-1)$ blows up as the final time $T$ is approached. If $\alpha=1/2$, by the pinching estimate \eqref{s4:pinc} we can check that the evolution equation of rescaled speed function $\Theta(t,T)\sqrt{K-1}$ is uniformly parabolic, then the Harnack inequality gives the lower bound on $\Theta(t,T)\sqrt{K-1}$, where as before $\Theta(t,T)$ denotes the spherical solution to \eqref{4.1} with the same maximum existence time $T$. The convergence of the rescaled solutions can be proved using a similar procedure as in \cite{Andrews-Chen2017}. For $1/2<\alpha\leq 1$, the evolution equaiton of the rescaled speed function is a porous medium type equation as in Lemma \ref{s7:lem3}. Then the argument given in the proof of case (i) can be adapted to complete the proof. The proof for case (iii) is similar.
\end{proof}

\begin{proof}[Proof of Theorem \ref{main-theo-II}]
The case (i) with $\alpha=1$ and case (ii) with $\alpha=1/2$ are included in the result by Gerhardt \cite{Ger15}. The case (ii) with $\alpha=1$ was proved by McCoy \cite{McCoy2017}. For the remaining cases, the evolution equation of the rescaled speed function can be written as a porous medium type equation as in Lemma \ref{s7:lem3}. The argument there can be adapted to complete the proof.
\end{proof}


\end{document}